\documentclass[12pt,psamsfonts]{amsart}
\usepackage{fullpage}
\usepackage{amssymb,amsfonts,amsmath}
\usepackage[all,cmtip]{xy}
\usepackage{enumerate}
\usepackage{mathrsfs}
\usepackage{comment}
\usepackage{hyperref} 
\usepackage[capitalise]{cleveref}
\usepackage{floatrow}
\usepackage{todonotes}
\usepackage{url}
\usepackage{constants}

\let\oldtocsection=\tocsection

\let\oldtocsubsection=\tocsubsection

\let\oldtocsubsubsection=\tocsubsubsection

\renewcommand{\tocsection}[2]{\hspace{0em}\oldtocsection{#1}{#2}}
\renewcommand{\tocsubsection}[2]{\hspace{1em}\oldtocsubsection{#1}{#2}}
\renewcommand{\tocsubsubsection}[2]{\hspace{2em}\oldtocsubsubsection{#1}{#2}}
\newtheorem{thm}{Theorem}[section]
\newtheorem{cor}[thm]{Corollary}
\newtheorem{prop}[thm]{Proposition}
\newtheorem{lem}[thm]{Lemma}

\newtheorem{quest*}{Question}
\newtheorem{prob*}{Problem}

\theoremstyle{definition}

\theoremstyle{remark}
\newtheorem{rem*}{Remark}
\newtheorem{rems*}[thm]{Remarks}

\numberwithin{equation}{section}
\newcounter{notation}

\DeclareUrlCommand\DOI{}
\newcommand{\doi}[1]{\href{http://dx.doi.org/#1}{\DOI{doi:#1}}}



\crefname{figure}{Figure}{Figures}
\theoremstyle{plain}
\newtheorem*{thm*}{Theorem}
\crefname{thm}{Theorem}{Theorems}
\crefname{cor}{Corollary}{Corollarys}
\newtheorem*{cor*}{Corollary}
\crefname{cor*}{Corollary}{Corollarys}
\crefname{lem}{Lemma}{Lemmas}
\crefname{prop}{Proposition}{Propositions}
\crefname{conj}{Conjecture}{Conjectures}
\newtheorem*{conj*}{Conjecture}
\crefname{conj*}{Conjecture}{Conjectures}
\crefname{defn}{Definition}{Definitions}

\theoremstyle{remark}
\newtheorem*{remark}{Remark}

\def\addsymbol #1: #2#3{$#1$ \> \parbox{5.4in}{#2 \dotfill \pageref{#3}}\\} 
\def\addsymbolEND #1: #2#3{$#1$ \> \parbox{5.4in}{#2 \dotfill \pageref{#3}}}

\renewcommand{\1}{ {\bf 1}}

\renewcommand{\bar}{\overline}

\newcommand{\kf}{\mathfrak{f}}

\newcommand{\sL}{\mathscr{L}}

\renewcommand{\Im}{\mathrm{Im}}

\renewcommand{\pmod}[1]{\, (\mathrm{mod} {\, #1})}
\newcommand{\kn}{\mathfrak{n}}
\newcommand{\N}{\mathrm{N}}
\newcommand{\cN}{\mathcal{N}}
\newcommand{\kN}{\mathfrak{N}}

\newcommand{\kp}{\mathfrak{p}}

\newcommand{\kP}{\mathfrak{P}}

\newcommand{\cO}{\mathcal{O}}
\newcommand{\ord}{\mathrm{ord}\,}

\newcommand{\Q}{\mathbb{Q}}
\newcommand{\cQ}{\mathcal{Q}}

\newcommand{\R}{\mathbb{R}}

\renewcommand{\Re}{\mathrm{Re}}

\newcommand{\sumP}{\sideset{}{'}\sum}
\newcommand{\sumS}{\sideset{}{^\star}\sum}

\newcommand{\Z}{\mathbb{Z}}
\newcommand{\cZ}{\mathcal{Z}}

\title[A Chebotarev variant of the Brun-Titchmarsh theorem]{A Chebotarev variant of the Brun-Titchmarsh theorem and bounds for the Lang-Trotter conjectures}

\author{Jesse Thorner}
\address{
Department of Mathematics, Stanford University \\
Building 380, Sloan Mathematical Center, Stanford, CA, 94305, United States}
\email{jesse.thorner@gmail.com}

\author{Asif Zaman}
\address{
Department of Mathematics, University of Toronto \\
Room 6290, 40 St. George St.,  Toronto, ON, M5S2E4, Canada}
\email{asif@math.toronto.edu}

\date{\today}

\begin{document}

\maketitle
\begin{abstract}
We improve the Chebotarev variant of the Brun-Titchmarsh theorem proven by Lagarias, Montgomery, and Odlyzko using the log-free zero density estimate and zero repulsion phenomenon for Hecke $L$-functions that were recently proved by the authors.  Our result produces an improvement for the best unconditional bounds toward two conjectures of Lang and Trotter regarding the distribution of traces of Frobenius for elliptic curves and holomorphic cuspidal modular forms.  We also obtain new results on the distribution of primes represented by positive-definite integral binary quadratic forms.
\end{abstract}

\section{Introduction and Statement of Results}

Let $\pi(x;q,a)$ denote the number of primes $p\leq x$ such that $p\equiv a\pmod q$.  The Siegel-Walfisz theorem states that if $(a,q)=1$ and there exists some constant $A>0$ such that $q\leq(\log x)^A$, then
\begin{equation}
\label{eqn:PNT_AP}
\pi(x;q,a)\sim\frac{1}{\varphi(q)}\mathrm{Li}(x),
\end{equation}
where $\mathrm{Li}(x)=\int_2^x \frac{dt}{\log t} \sim \frac{x}{\log x}$.  Assuming the generalized Riemann hypothesis, the range of $q$ extends to $q\leq x^{1/2-\epsilon}$ for any $\epsilon>0$.  Any unconditional improvement in the range of $q$ wonld preclude the existence of a real Landau-Siegel zero for the $L$-functions of real Dirichlet characters.  Since this seems to be beyond the reach of current techniques, it is often useful to trade asymptotic equality in \eqref{eqn:PNT_AP} for upper and lower bounds of the correct asymptotic order which hold in improved ranges of $x$.

The first lower bound of this form follows from Fogels' improvements \cite{Fogels} to the ideas of Linnik \cite{Linnik}.  These ideas were substantially improved by Heath-Brown \cite{HBLinnik} and Maynard \cite{Maynard}, the latter of whom proved that if $q$ is sufficiently large\footnote{All implied constants in this paper are effectively computable.  Unless specifically mentioned otherwise, all implied constants in this paper are also absolute.}, then
\begin{equation}
\label{eqn:Linnik}
\pi(x;q,a)\gg \frac{\log q}{\varphi(q)\sqrt{q}}\frac{x}{\log x}\qquad \text{for } x\geq q^{8}.
\end{equation}
To describe upper bounds in improved ranges of $q$, we define $\theta=(\log q)/\log x$.  Titchmarsh \cite{Titchmarsh} used Brun's sieve to show that if $\theta<1$, then
\begin{equation}
\label{eqn:BT}
\pi(x;q,a)\ll\frac{1}{1-\theta}\frac{x}{\varphi(q)\log x}.
\end{equation}
The implied constant can be made explicit, and has been estimated by various authors.  The strongest result in this direction for all ranges of $q$ is due to Montgomery and Vaughan \cite{MV}; they used the large sieve inequality to prove that if $\theta<1$, then
\begin{equation}
\label{eqn:MV}
\pi(x;q,a)\leq\frac{2}{1-\theta}\frac{x}{\varphi(q)\log x}.
\end{equation}

Since the factor of 2 is unlikely to be improved using current techniques, many authors have improved the $\theta$-dependence.  To summarize, if $q$ is sufficiently large, then
\[
\pi(x;q,a)\leq\frac{(C(\theta)+o(1))}{\varphi(q)}\frac{x}{\log x},
\]
where
\begin{equation}
C(\theta)=\begin{cases}
(2-(\frac{1-\theta}{4})^6)/(1-\theta)&\mbox{if $2/3\leq\theta<1$,}\\
8/(6-7\theta)&\mbox{if $9/20<\theta< 2/3$},\\
16/(8-3\theta)&\mbox{if $1/8<\theta\leq 9/20$},\\
2&\mbox{if $0<\theta\leq 1/8$},
\end{cases}
\label{eqn:ClassicalBT}
\end{equation}
the last line being recently proven by Maynard \cite{Maynard}.  (See \cite{Maynard} and the sources contained therein for a thorough overview of the problem.)  While progress on \eqref{eqn:BT} has typically followed from advances in sieve theory and exponential sums, Maynard's proof builds on Heath-Brown's analysis in \cite{HBLinnik} and uses a log-free zero density estimate for Dirichlet $L$-functions and careful analysis of Landau-Siegel zeros.

In this paper, we consider analogous questions for the distribution of prime ideals in the context of the Chebotarev density theorem.  Let $L/F$ be a finite Galois extension of number fields with Galois group $G$, and let $C\subset G$ be a conjugacy class.  Let $D_L$ denote the absolute value of the discriminant of $L/\Q$.  To each prime ideal $\kp$ of $F$ that does not ramify in $L$, there corresponds a certain conjugacy class of automorphisms in $G$ which are attached the prime ideals of $L$ lying above $\kp$.  We denote this conjugacy class by the Artin symbol $[\frac{L/F}{\kp}]$.  For a fixed conjugacy class $C\subset G$, let
\begin{equation}
\pi_C(x,L/F):=\#\Big\{\kp:\textup{$\kp$ unramified in $L$, $\Big[\frac{L/F}{\kp}\Big]=C$, $\mathrm{N}_{F/\Q}\kp\leq x$}\Big\}.
\label{def:pi_C}
\end{equation}
The Chebotarev density theorem, in the effective version proven by Lagarias and Odlyzko \cite{LO}, states that if $x\geq \exp(10[L:\Q](\log D_L)^2)$, then
\begin{equation}
\label{eqn:cdt_LO}
\pi_C(x,L/F)\sim\frac{|C|}{|G|}\mathrm{Li}(x).
\end{equation}
This subsumes many results in the distribution of primes, including the distribution of quadratic nonresidues modulo $D$ for any $D$, primes in arithmetic progressions, and prime ideals for any number field.  As such, we are interested in upper and lower bounds of $\pi_C(x,L/F)$ of the correct order of magnitude with an improved range of $x$.

A lower bound on $\pi_C(x,L/F)$ with the correct order of magnitude (in the $x$-aspect) follows from the work of Weiss \cite{Weiss}, which was recently made explicit by Thorner and Zaman \cite{JT_AZ}.  Let $H\subset G$ be a largest abelian subgroup such that $H\cap C$ is nonempty, and let $K$ be the fixed field of $H$.  For a character $\chi$ in the dual group $\widehat{H}$, let $\mathfrak{f}_{\chi}$ be the conductor of $\chi$, and define
\begin{equation}
\label{eqn:Qmax_def}
\mathcal{Q}(L/K)=\max\{\mathrm{N}_{K/\Q}\mathfrak{f}_{\chi}:\chi\in\widehat{H}\}.
\end{equation}
Thorner and Zaman proved that if $x\geq D_K^{694}\mathcal{Q}(L/K)^{521}+D_K^{232}\mathcal{Q}(L/K)^{367}[K:\Q]^{290[K:\Q]}$, then
\[
\pi_C(x,L/F)\gg\frac{1}{(D_K \mathcal{Q}(L/K)[K:\Q]^{[K:\Q]})^5}\frac{x}{[L:K]\log x},
\]
provided that $D_K\mathcal{Q}(L/K)[K:\Q]^{[K:\Q]}$ is sufficiently large.  When this is applied to arithmetic progressions (in which case $L=\Q(e^{2\pi i /q})$ for $q$ sufficiently large and $F=K=\Q$), this yields the bound
\[
\pi(x;q,a)\gg\frac{1}{q^5}\frac{x}{\varphi(q)\log x}\quad \text{for }  x\geq q^{521}.
\]
Up to the quality of the exponents, this is comparable to \eqref{eqn:Linnik}.

In analogy with \eqref{eqn:BT}, Lagarias, Montgomery, and Odlyzko \cite{LMO} proved that
\begin{equation}
\label{eqn:BT_CDT}
\pi_C(x,L/F)\ll \frac{|C|}{|G|}\mathrm{Li}(x), \qquad \log x\gg(\log D_L)(\log\log D_L)(\log\log\log e^{20}D_L).
\end{equation}
(Serre \cite{Serre} showed that $e^{20}$ can be replaced with 6.) There are several large sieve inequalities yielding Brun-Titchmarsh type results for counting prime integers in the ring of integers of a number field (e.g., \cite{Huxley_LargeSieve_I,Schaal}) and for counting prime ideals lying in arithmetic progressions (e.g., \cite{HinzLodemann}), but it appears that \eqref{eqn:BT_CDT} is the only Brun-Titchmarsh type bound that counts prime ideals with effective field dependence.  While the range of $x$ in \eqref{eqn:BT_CDT} is noticeably less restrictive than the range of $x$ for which \eqref{eqn:cdt_LO} holds, the range still depends poorly on $L$; this can be prohibitive for many applications.  It does not seem to be the case that sieve methods can produce a range of $x$ that is comparable to \eqref{eqn:MV}.  Using the log-free zero density estimate and zero repulsion results proved by Thorner and Zaman in \cite{JT_AZ}, we improve the range of $x$ in \eqref{eqn:BT_CDT}.

\begin{thm}
\label{thm:BT_Weiss}
Let $L/F$ be a Galois extension of number fields with Galois group $G$ with $L \neq \Q$. Let $C$ be any conjugacy class of $G$ and let $H$ be an abelian subgroup of $G$ such that $H \cap C$ is non-empty. If $K$ is the subfield of $L$ fixed by $H$ and $\mathcal{Q} = \mathcal{Q}(L/K)$ is given by \eqref{eqn:Qmax_def}, then
\[
\pi_C(x,L/F) \ll \frac{|C|}{|G|} \mathrm{Li}(x),
\]
provided that
\begin{equation}
x\gg D_K^{246} \mathcal{Q}^{185} + D_K^{82} \mathcal{Q}^{130} [K:\Q]^{246 [K:\Q]}.
\label{eqn:range_BT_Weiss}
\end{equation}
\end{thm}

\begin{remark}
	For the valid range of $x$, one can minimize the exponents of $D_K$ and $\cQ$ at the expense of a less desirable dependence on $[K:\Q]^{[K:\Q]}$ and vice versa. In particular, the same upper bound for $\pi_C(x, L/F)$ holds when
	\begin{equation}
	\label{eqn:range_BT_Weiss_2}
	x \gg D_K^{164} \mathcal{Q}^{123} + D_K^{55} \mathcal{Q}^{87} [K:\Q]^{68 [K:\Q]} + D_K^2 \cQ^2 [K:\Q]^{14,000[K:\Q]}.
	\end{equation}
	See the remarks at the end of \cref{sec:Proof_BT_easy} for details.
\end{remark}

Our result always gives an improvement over \eqref{eqn:BT_CDT}.  Choosing $H$ to be the cyclic group generated by a fixed element of $C$, we have that $D_L^{1/|H|}\leq D_K\mathcal{Q}\leq D_L^{1/\varphi(|H|)}$ (see \cite[Section 6]{Weiss})  Moreover, by the classical work of Minkowski, we have that $[K:\Q] \ll \log D_K \leq \log D_L$. Therefore, \cref{thm:BT_Weiss} holds when $\log x \gg (\log D_L)(\log\log D_L)$, which is a modest unconditional improvement over \eqref{eqn:BT_CDT}.  However, one usually obtains a more significant improvement.  For most fields $K$, the bound $[K:\Q]\ll (\log D_K)/\log\log D_K$ holds.  In this case, we may take $\log x\gg \log(D_K\mathcal{Q})$ in \cref{thm:BT_Weiss}.    Thus  \cref{thm:BT_Weiss} holds when $\log x\gg (\log D_L)/\varphi(|H|)$, which noticeably improves \eqref{eqn:BT_CDT}.

Building on \cite{Maynard}, we obtain an implied constant that is essentially sharp (short of precluding the existence of Landau-Siegel zeros) when $x$ is sufficiently large in terms of $L/F$.

\begin{thm}
\label{thm:BT_Weiss_sharp}
Let $L/F$ be a Galois extension of number fields with Galois group $G$ and let $C$ be any  conjugacy class of $G$. Let $H$ be an abelian subgroup of $G$ such that $H \cap C$ is non-empty. If $K$ is the subfield of $L$ fixed by $H$ and $\mathcal{Q} = \mathcal{Q}(L/K)$ is given by \eqref{eqn:Qmax_def}, then 
\[
\pi_C(x,L/F) < \Big\{ 2 + O\Big([K:\Q] x^{-\tfrac{1}{166[K:\Q]+327}} \Big) \Big\} \frac{|C|}{|G|} \mathrm{Li}(x)
\]
for 
\begin{equation}
	\label{eqn:range_BT_Weiss_sharp}
	x \gg D_K^{695} \mathcal{Q}^{522} + D_K^{232} \mathcal{Q}^{367} [K:\Q]^{290 [K:\Q]},
\end{equation}
provided that $D_K \mathcal{Q} [K:\Q]^{[K:\Q]}$ is sufficiently large.  If any of the following conditions also hold, then the error term can be omitted:
	\begin{itemize}
		\item There exists a sequence of number fields $\Q=K_0\subset K_1\subset\cdots\subset K_n=K$ such that $K_{j+1}/K_j$ is a normal extension for all $j=0,1,\ldots,n-1$.
		\item $(2[K:\Q])^{2[K:\Q]^2} \ll D_K \cQ^{1/2}$.
		\item $x \gg [K:\Q]^{334[K:\Q]^2}$.
	\end{itemize}
\end{thm}

In the special case where $L/\Q$ is an abelian Galois extension, we may take $K=\Q$ in Theorem \ref{thm:BT_Weiss_sharp}.  Since $\Q/\Q$ is trivially a normal extension, the error term in \cref{thm:BT_Weiss_sharp} can be omitted, and we recover Maynard's result in \eqref{eqn:ClassicalBT} for $\theta \leq 1/522$. (See the remark at the end of \cref{sec:Proof_BT_SiegelZeroExists} for details.)  Another interesting set of primes for which the normal tower condition in \cref{thm:BT_Weiss_sharp} applies is the set of primes represented by binary quadratic forms.  Suppose $Q(X,Y)$ is a positive-definite primitive binary integral quadratic form with discriminant $-D$. It  is well-known  that such forms, up to $\mathrm{SL}_2$-equivalence, form a group which is isomorphic to the ring class group of the imaginary quadratic field $\Q(\sqrt{-D})$ (see \cite[Theorem 7.1]{Cox} for example). Further, a rational prime $p$ is represented by $Q(X,Y)$ if and only if there exists a prime ideal $\kp$ in $\Q(\sqrt{-D})$ such that its norm equals $p$ and $\kp$ belongs to the corresponding class of $Q(X,Y)$. It follows by the Chebotarev density theorem that
\begin{equation}
\#\{ p \leq x   : p \text{ is represented by } Q(X,Y) \} \sim \delta_Q\frac{\mathrm{Li}(x)}{h(-D)} \qquad \text{as } x \rightarrow \infty,
\label{eqn:PrimesOfQuadForms_Asymptotic}
\end{equation}
where $\delta_Q=1/2$ if $Q(X,Y)$ is properly equivalent to its opposite and $\delta_Q=1$ otherwise, and $h(-D)$ is the number of such forms of discriminant $-D$ up to $\mathrm{SL}_2$-equivalence.  To obtain an upper bound for the number of such primes, we let $F = \Q(\sqrt{-D})$, and we let $L$ be the ring class field of the order of the discriminant $-D$.  Thus $\mathrm{Gal}(L/F)$ is abelian.  Applying \eqref{eqn:range_BT_Weiss_2} and \cref{thm:BT_Weiss_sharp} to $L/F$, with $C$ equal to the singleton conjugacy class in $G$ corresponding to $Q(X,Y)$, we obtain the following.

\begin{cor} 
\label{cor:QuadForms}
Let $Q(X,Y)$ be a positive-definite primitive binary integral quadratic form with discriminant $-D$, and let $h(-D)$ be the number of such quadratic forms up to $\mathrm{SL}_2$-equivalence. For $x \gg D^{164}$,
\begin{equation}
\#\{ p \leq x\colon p \textup{ is represented by } Q(X,Y) \} \ll \frac{\mathrm{Li}(x)}{h(-D)}
\label{eqn:QuadForms_Bound1}
\end{equation}
with an absolute implied constant.  Also, if $D$ is sufficiently large, then for $x \gg D^{695}$,
\[
\#\{ p \leq x\colon p \textup{ is represented by } Q(X,Y) \} < 2\delta_Q \,\frac{\mathrm{Li}(x)}{h(-D)},
\]
where $\delta_Q=1/2$ if $Q(X,Y)$ is properly equivalent to its opposite and $\delta_Q=1$ otherwise.
\end{cor}

\begin{remark}
Note that \eqref{eqn:BT_CDT} also implies \eqref{eqn:QuadForms_Bound1} in the much more restricted range $x \gg D^{O(D^{1/2+\epsilon})}$ for any fixed $\epsilon > 0$.  On the other hand, \cref{cor:QuadForms} gives the range $x\gg D^{O(1)}$, which is comparable (up to the quality of the exponent) to the range $x \gg D^{1+\epsilon}$ predicted by the generalized Riemann hypothesis for Hecke $L$-functions.
\end{remark}

We use \cref{thm:BT_Weiss} to improve the best unconditional upper bounds for two outstanding conjectures of Lang and Trotter \cite{LT}.  Let
\begin{equation}
\label{eqn:newform}
f(z)=\sum_{n=1}^{\infty} a_f(n)e^{2\pi i n z}
\end{equation}
be a holomorphic cusp form of even integral weight $k_f\geq2$ and level $N_f$; for simplicity, we assume that $a_f(n)\in\Z$ for all $n\geq1$.  Suppose that $f$ does not have complex multiplication, that the nebentypus of $f$ is trivial, and that $f$ is a newform (i.e., $f$ is a normalized eigenform for the Hecke operators $T_p$ for $p\nmid N_f$ and $U_p$ for $p\mid N_f$).  Fix $a\in\Z$, and let
\begin{equation}
\label{eqn:pi_f}
\pi_f(x,a)=\#\{p\leq x :a_f(p)=a\}.
\end{equation}
Lang and Trotter conjectured that as $x\to\infty$, we have that
\[
\pi_f(x,a)\sim c_{f,a}\begin{cases}
\sqrt{x}(\log x)^{-1}&\mbox{if $k_f=2$},\\
1&\mbox{if $k_f\geq4$},
\end{cases}
\]
where $c_{f,a}\geq0$ is a certain constant depending on $f$ and $a$ alone.

In the special case where $k_f=2$, Elkies \cite{Elkies} proved that $\pi_f(x,0)\ll_{N_f} x^{3/4}$.  In all other cases, Serre proved in 1981 that
\[
\pi_f(x,a)\ll_{N_f}\frac{x}{(\log x)^{1+\delta}}
\]
for any $\delta<1/4$; following the ideas of M. R. Murty, V. K. Murty, and Saradha \cite{MMS}, Wan \cite{Wan} improved the range of $\delta$ in 1990 to any $\delta<1$.  This was further sharpened by V. K. Murty \cite{VKM} in 1997; he proved\footnote{Theorem 5.1 of \cite{VKM} actually claims a stronger result, but a step in the proof seems not to be justified. The best that the argument appears to give is what we have stated above; see \cref{sec:LT_proofs_1} for details.} that
\begin{equation}
\label{eqn:LT_111}
\pi_f(x,a)\ll_{N_f} \frac{x(\log \log x)^3}{(\log x)^2}.
\end{equation}
Using \cref{thm:BT_Weiss}, we give a modest improvement\footnote{Note that we recover Murty's claimed result \cite[Theorem 5.1]{VKM}.}.
\begin{thm}
\label{thm:LT_1}
Let $f$ be a newform of even integral weight $k_f\geq2$, level $N_f$, and trivial nebentypus with integral coefficients. If $\pi_f(x,a)$ is given by \eqref{eqn:pi_f}, then
\[
\pi_f(x,a)\ll_{N_f}\frac{x(\log\log x)^2}{(\log x)^2}.
\]
\end{thm}

We also consider a different (but closely related) conjecture of Lang and Trotter regarding the Frobenius fields of an elliptic curve.  Let $E/\Q$ be an elliptic curve of conductor $N_E$ without complex multiplication.  For a prime $p\nmid N$, let $\Pi_p$ be the Frobenius endomorphism of $E/\mathbb{F}_p$.  Defining $a_E(p)=p+1-\#E(\mathbb{F}_p)$, we have that $\Pi_p^2-a_E(p)\Pi_p+p=0$.  By Hasse, we know that $|a_E(p)|<2\sqrt{p}$, so $\Q(\Pi_p)$ in $\mathrm{End}(E/\mathbb{F}_p)\otimes_{\Z}\Q$ is an imaginary quadratic field.  For a fixed imaginary quadratic field $k$ with absolute discriminant $D_k$, let
\begin{equation}
\label{eqn:pi_E}
\pi_E(x,k)=\#\{p\leq x:\Q(\Pi_p)\cong k\}.
\end{equation}
Lang and Trotter conjectured that as $x\to\infty$,
\[
\pi_E(x,k)\sim c_{E,k}\frac{\sqrt{x}}{\log x},
\]
where $c_{E,k}$ is a certain constant depending on $E$ and $k$ alone.  Using the square sieve, Cojocaru, Fouvry, and M. R. Murty \cite{CFM} proved that
\[
\pi_E(x,k)\ll_{N_E,k}\frac{x(\log\log x)^{13/12}}{(\log x)^{25/24}}.
\]
Using V. K. Murty's version of the Chebotarev density theorem and Serre's method of mixed representations (see \cite{Serre}), Zywina \cite{Zywina} improved this bound to
\begin{equation}
\label{eqn:LT_2}
\pi_E(x,k)\ll_{N_E,k} \frac{x(\log \log x)^2}{(\log x)^2}.
\end{equation}
Using \cref{thm:BT_Weiss}, we establish a modest improvement to \eqref{eqn:LT_2}.
\begin{thm}
\label{thm:LT_2}
Let $E/\Q$ be an elliptic curve of conductor $N_E$ and let $k$ be a fixed imaginary quadratic number field. If $\pi_E(x,k)$ is defined by \eqref{eqn:pi_E} then
\[
\pi_E(x,k)\ll_{N_E,k}\frac{x\log\log x}{(\log x)^2}.
\]
\end{thm}

\begin{remark}
A similar infinite Galois extension problem is described by Theorem 10 in Section 4.1 of \cite{Serre}, and \cref{thm:BT_Weiss} gives a similar improvement.
\end{remark}

In Sections \ref{sec:Preliminaries}-\ref{sec:NegligibleZeros}, we discuss necessary results on Hecke $L$-functions and provide the analytic setup for the proofs of Theorems \ref{thm:BT_Weiss} and \ref{thm:BT_Weiss_sharp}.  These results are then proved in \cref{sec:Proof_BT_easy} and Sections \ref{sec:Proof_BT_SiegelZeroExists}-\ref{sec:Proof_BT_NoSiegelZero}, respectively.  Finally, we prove Theorems \ref{thm:LT_1} and \ref{thm:LT_2} in Section \ref{sec:LT_proofs_1}.

\subsection*{Acknowledgements}
The authors thank John Friedlander, V. K. Murty, and Ken Ono for their comments and Tristan Freiberg for starting our interest in the Brun-Titchmarsh problem for number fields. The first author is supported by a NSF Mathematical Sciences Postdoctoral Research Fellowship.

\section{Initial Setup}
\label{sec:Preliminaries}

\subsection{Notation} 
\label{subsec:notation}
For a number field $F$, we will use the following notation throughout:
\begin{itemize}
	\item $\mathcal{O}_F$ is the ring of integers of $F$.
	\item $n_F = [F:\mathbb{Q}]$ is the degree of $F/\mathbb{Q}$.
	\item $D_F = |\mathrm{disc}(F/\Q)|$ is the absolute value of the discriminant of $F$.
	\item $\N_{F/\mathbb{Q}}$ is the absolute field norm of $F$. 
	\item $\zeta_F(s)$ is the Dedekind zeta function of $F$.
	\item $\kp$ is a prime ideal of $F$. 
	\item $\kn$ is an integral ideal of $F$. 
	\item $\Lambda_F(\kn)$ is the von Mangoldt $\Lambda$-function for $F$ given by
\begin{equation*}
\Lambda_F(\mathfrak{n}) = 
	\begin{cases}
		\log \N_{F/\Q}\mathfrak{p} & \text{if $\mathfrak{n}$ is a power of a prime ideal $\mathfrak{p}$,} \\
		0 & \text{otherwise.}
	\end{cases}
	\label{def:vonMangoldt}
\end{equation*}
\end{itemize}
If it is clear from context, we will write $\N = \N_{F/\Q}$ for convenience. 

We also adhere to the convention that all implied constants in all asymptotic inequalities $f\ll g$ or $f=O(g)$ are absolute.  If an implied constant depends on a field-independent parameter, such as $\epsilon$, then we use $\ll_{\epsilon}$ and $O_{\epsilon}$ to denote that the implied constant depends at most on $\epsilon$.  All implied constants will be effectively computable.

\subsection{Prime ideal counting functions}
\label{subsec:PrimeCounters}

We briefly recall the definition of an Artin $L$-function from \cite[Chapter 2, Section 2]{MR1482805}.  Let $L/F$ be a Galois extension of number fields with Galois group $G$.  For each prime ideal $\kp$ of $F$, and a prime ideal $\kP$ of $L$ lying above $\kp$, we define the decomposition group $D_{\kP}$ to be $\mathrm{Gal}(L_{\kP}/F_{\kp})$, where $L_{\kP}$ (resp. $K_{\kp}$) is the completion of $L$ (resp. $K$) at $\kP$ (resp. $\kp$).  We have a map $D_{\kP}$ to $\mathrm{Gal}(k_{\kP}/k_{\kp})$ (the Galois group of the residue field extension), which is surjective by Hensel's lemma.  The kernel of this map is the inertia group $I_{\kP}$.  We thus have the exact sequence
\[
1\to I_{\kP}\to D_{\kP}\to\mathrm{Gal}(k_{\kP}/k_{\kp})\to 1.
\]
The group $\mathrm{Gal}(k_{\kP}/k_{\kp})$ is cyclic with generator $x\mapsto x^{\mathrm{N}\kp}$, where $\mathrm{N}\kp$ is the cardinality of $k_{\kp}$.  We can choose an element $\sigma_{\kP}\in D_{\kP}$ whose image in $\mathrm{Gal}(k_{\kP}/k_{\kp})$ is this generator.  We call $\sigma_{\kP}$ a Frobenius element at $\kP$; it is well-defined modulo $I_{\kP}$.  We have that $I_{\kP}$ is trivial for all unramified $\kp$, and for these $\kp$, $\sigma_{\kP}$ is well-defined.  For $\kp$ unramified, we denote by $\sigma_{\kp}$ the conjugacy class of Frobenius elements at primes $\kP$ above $\kp$.

Let $\rho:G\to\mathrm{GL}_n(\mathbb{C})$ be a representation of $G$, and let $\psi$ denote its character.  Let $V$ be the underlying complex vector space on which $\rho$ acts, and let $V^{I_{\kP}}$ be the subspace of $V$ on which $I_{\kP}$ acts trivially.  We now define
\[
L_{\kp}(s,\psi,L/F)=\begin{cases}
\det(I_n - \rho(\sigma_{\kp})\mathrm{N}\kp^{-s})^{-1}&\mbox{if $\kp$ is unramified in $L$},\\
\det(I_n-\rho(\sigma_{\kP})\mid_{V^{I_{\kP}}}\mathrm{N}\kp^{-s})^{-1}&\mbox{if $\kp$ is ramified in $L$}.
\end{cases}
\]
This is well-defined for all $\kp$, which allows us to define the Artin $L$-function
\[
L(s,\psi,L/F)=\prod_{\kp}L_{\kp}(s,\psi,L/F)
\]
for $\Re\{s\} > 1$. Now, for a conjugacy class $C \subseteq G$, let $g_C \in C$ be arbitrary. Define
\begin{equation}
Z_C(s) := - \frac{|C|}{|G|} \sum_{\psi} \bar{\psi}(g_C) \frac{L'}{L}(s,\psi,L/F),
\label{def:ZC}
\end{equation}
where $\psi$ runs over irreducible characters of $G$ and $L(s,\psi,L/F)$ is the associated Artin $L$-function. Note the definition of $Z_C(s)$ does not depend on the choice of $g_C$ since $\psi$ is the trace of the representation $\rho$ and $g_C$ is conjugate to any other choice. By orthogonality relations for characters (see \cite[Section 3]{Heilbronn} for example),
\begin{equation}
Z_C(s) = \sum_{\kn \subseteq \cO_F} \Lambda_F(\kn) \Theta_C(\kn)  (\N\kn)^{-s},
\label{eqn:ZC_orthogonality}
\end{equation}
where $\Theta_C(\kn)$ is supported on integral ideals $\kn$ which are powers of a prime ideal; in particular, for prime ideals $\kp$ unramified in $L$ and $m \geq 1$, 
\begin{equation}
	\Theta_C(\kp^m) 
	 = \begin{cases}
 			1 & \text{if $[\frac{L/F}{\kp}]^m \subseteq C$,} \\
 			0 & \text{otherwise,}
		 \end{cases}
		 \label{def:Indicator_Conjugacy}
\end{equation}
and  $0 \leq \Theta_C(\kp^m) \leq 1$ if $\kp$ ramifies in $L$. (This discussion and definition of $\Theta_C(\, \cdot \,)$ is also contained in \cite[Section 3]{LMO}.)  For $x > 1$, define
\begin{equation}
		\psi_C(x) := \sum_{\substack{\N\kn < x} }  \Lambda_F(\kn) \Theta_C(\kn),
		\label{def:PrimeCountingFunctions}
\end{equation}
where the sum is over integral ideals $\kn$ of $F$. By standard arguments, this prime ideal counting function is related to $\pi_C(x, L/F)$ given by \eqref{def:pi_C}. Since we are only interested in an upper bound for $\pi_C(x,L/F)$, we give a simpler statement that suffices for our purposes. 
\begin{lem}
	\label{lem:Pi_to_Psi}
	If $x > x_0 > 3$, then
	\begin{equation*}
	\pi_C(x,L/F) \leq \frac{\psi_C(x)}{\log x} + \int_{x_0}^x \frac{\psi_C(t)}{t \log^2 t} dt + O(n_F x_0). 
	\end{equation*}
\end{lem}

\begin{proof}

Let $t>1$.  We define
\[
\tilde{\pi}_C(t) := \sum_{\substack{\N\kp < t} } \Theta_C(\kp), \qquad \theta_C(t) := \sum_{\N\kp < t} \Theta_C(\kp) \log \N\kp,
\]
where the sums are over all prime ideals $\kp$ of $F$. First, observe that, by \eqref{def:Indicator_Conjugacy}, the only difference between $\tilde{\pi}_C(x)$ and $\pi_C(x,L/F)$ is the contribution from the prime ideals $\kp$ of $F$ ramified in $L$. Since $0 \leq \Theta_C(\kp) \leq 1$ for such prime ideals, we observe that
\begin{equation}
\pi_C(x,L/F) \leq \tilde{\pi}_C(x),
\label{eqn:Pi_to_TildePi}
\end{equation}
so it suffices to estimate $\tilde{\pi}_C(x)$. Using partial summation, we see that if $3 < x_0 < x$, then
\begin{equation}
\label{eqn:TidlePi_PartialSummation}	\tilde{\pi}_C(x) = \frac{ \theta_C(x) }{\log x} + \int_{x_0}^x \frac{\theta_C(t)}{t \log^2 t}dt + \tilde{\pi}_C(x_0). 
\end{equation}
Since there are at most $n_F$ prime ideals above a rational prime $p$, observe that
\begin{equation}
\tilde{\pi}_C(x_0) \leq \sum_{p < x_0} \sum_{\kp \mid (p)} 1 \leq n_F \sum_{p < x_0} 1 \ll \frac{n_F x_0}{\log x_0} \ll n_F x_0. 
\label{eqn:Pi_Bound}
\end{equation}
Moreover, $\theta_C(t) \leq \psi_C(t)$ for all $t > 1$. Combining these observations with \eqref{eqn:Pi_to_TildePi} and \eqref{eqn:TidlePi_PartialSummation} yields the desired result.  
\end{proof}


\subsection{Choice of Weight}
Let us define a weight function and describe its properties. This choice of weight can be regarded as a smoothed version of Maynard's weight \cite[Equation (5.6)]{Maynard}. It will be used to count prime ideals with norm between $x^{1/2}$ and $x$.

\begin{lem} \label{lem:WeightChoice}

For any $x \geq 3, \epsilon \in (0,1/4)$, and positive integer $\ell \geq 1$, select
\[
A = \frac{\epsilon}{2 \ell \log x}. 
\]
There exists a real-variable function $f(t)  = f(t; x, \ell, \epsilon)$ such that:
\begin{enumerate}[(i)]
	\item $0 \leq f(t) \leq 1$ for all $t \in \R$, and $f(t) \equiv 1$ for $\tfrac{1}{2} \leq t \leq 1$.
	\item The support of $f$ is contained in the interval $[\tfrac{1}{2} - \frac{\epsilon}{\log x}, 1 +  \frac{\epsilon}{\log x}]$. 
	\item Its Laplace transform $F(z) = \int_{\R} f(t) e^{-zt}dt$ is entire and is given by
			\begin{equation}	
				F(z) = e^{-(1+ 2\ell A)z} \cdot \Big( \frac{1-e^{(\frac{1}{2}+2\ell A)z}}{-z} \Big) \Big( \frac{1-e^{2Az}}{-2Az} \Big)^{\ell}.
				\label{eqn:WeightLaplace}
			\end{equation}
	\item Let $s = \sigma + i t \in \mathbb{C}, \sigma > 0$ and $\alpha$ be any real number satisfying $0 \leq \alpha \leq \ell$. Then
	\[
	|F(-s\log x)|\leq 
		\displaystyle\frac{e^{\sigma \epsilon} x^{\sigma}}{|s| \log x} \cdot \big( 1 + x^{-\sigma/2} \big) \cdot  \Big( \frac{2\ell}{\epsilon|s|} \Big)^{\alpha}. 
	\]
	\item If $s = \sigma + i t \in \mathbb{C}$ and $\sigma > 0$, then
	\[
	|F(-s\log x)| \leq e^{\sigma \epsilon} x^{\sigma}.
	\]
		  Moreover, 
		  \[
		  1/2 < F(0) < 3/4, \qquad F(-\sigma\log x) \leq  \frac{e^{\epsilon} x^{\sigma} }{\sigma \log x}. 
		  \]
	\item Let $s = -\tfrac{1}{2}+it \in \mathbb{C}$. Then
	\[
	|F(-s\log x)| \leq  \frac{5 x^{-1/4}}{\log x} \Big( \frac{2\ell}{\epsilon}\Big)^{\ell} (1/4+t^2)^{-\ell/2}.
	\]
\end{enumerate}
\end{lem}
\begin{remark}
	Our choice is motivated by the works of Weiss \cite[Lemma 3.2]{Weiss} and the authors \cite[Lemma 9.1]{JT_AZ} on the least prime ideal. Namely, the weight function $f$ depends on a parameter $\ell$ which will be chosen to be of size $O(n_K)$. This forces $f$ to be $O(n_K)$-times differentiable and hence $F(x+iy)$ will decay like $|y|^{-O(n_K)}$ for fixed $x > 0$ and $|y| \rightarrow \infty$. This decay rate will be necessary when applying log-free zero density estimates such as \cref{thm:LFZD_HighLying} to bound the contribution of zeros which are high in the critical strip. 
\end{remark}

\begin{proof}
~
\begin{itemize}
	\item For parts (i) and (ii), let $\1_S(\, \cdot \,)$ be an indicator function for the set $S \subseteq \R$. For $j\geq1$, define 
\[
w(t) := \frac{1}{2 A} \mathbf{1}_{[-A,A]}(t), \quad g_0(t) := \mathbf{1}_{[\frac{1}{2} - \ell A,1 + \ell A]}(t), \quad \text{and} \quad
g_j(t) := (w \ast g_{j-1})(t).
\]
Since $\int_{\R} w(t) dt = 1$, one can verify that $f = g_{\ell}$ satisfies (i) and (ii).
	\item For part (iii), observe the Laplace transform $W(z)$ of $w$ is given by
\[
W(z) = \frac{e^{Az} - e^{-Az}}{2A z} = e^{-Az} \cdot \Big( \frac{1-e^{2Az}}{-2Az} \Big), 
\]
and the Laplace transform $G_0(z)$ of $g_0$ is given by
\[
G_0(z) =  \frac{e^{-(1/2-\ell A)z} - e^{-(1+\ell A)z}}{z} = e^{-(1+\ell A)z} \cdot \Big(\frac{1-e^{(\frac{1}{2}+2\ell A) z}}{-z}\Big).
\]
Thus (iii) follows as $F(z) = G_0(z) \cdot  W(z)^{\ell}$. 
	\item For part (iv), we see by (iii) and the definition of $A$ that 
\begin{equation}
|F(-s\log x)| \leq  \frac{e^{\sigma\epsilon} x^{\sigma}}{|s| \log x} \cdot \big( 1 + e^{-\sigma\epsilon} x^{-\sigma/2} \big) \Big| \frac{1-e^{-2A s \log x}}{2A s \log x} \Big|^{\ell}. 
\label{eqn:WeightChoice_iv}
\end{equation}
To bound the above quantity, we observe that
\begin{equation}
\Big|\frac{1-e^{-w}}{w}\Big|^2 \leq \Big(\frac{1-e^{-a}}{a}\Big)^2 \leq 1
\label{eqn:WeightChoice_ElementaryFunction}
\end{equation}
for $w = a+ib$ with $a > 0$ and $b \in \R$. This observation can be checked in a straightforward manner. Using \eqref{eqn:WeightChoice_ElementaryFunction}, it follows that
\begin{equation*}
\begin{aligned}
\Big| \frac{1-e^{-2A s \log x}}{2A s \log x} \Big|^{\ell} 
& = \Big| \frac{1-e^{-2A s \log x}}{2A s \log x} \Big|^{ \alpha } \cdot \Big| \frac{1-e^{-2A s \log x}}{2A s \log x} \Big|^{\ell -\alpha} 
& \leq \Big( \frac{1+ x^{-2A \sigma}}{2A |s| \log x}\Big)^{\alpha} \cdot 1 
& \leq \Big( \frac{2\ell}{\epsilon |s|}\Big)^{\alpha}. 
\end{aligned}
\end{equation*}
In the last step, we noted $1+x^{-2A\sigma} \leq 2$ and used the definition of $A$. Combining this with \eqref{eqn:WeightChoice_iv} and observing $e^{-\sigma \epsilon} \leq 1$, we deduce the desired bound.  

\item For part (v), we see by (iii) that
\begin{align*}
|F(-s\log x)| & \leq   \Big(\frac{1}{2}+2\ell A \Big) e^{\sigma\epsilon} x^{\sigma}  \cdot \Big|\frac{1- e^{-(\tfrac{1}{2}+2\ell A) s \log x}}{(\tfrac{1}{2}+2\ell A)s \log x} \Big|  \cdot \Big| \frac{1-e^{-2A s \log x}}{2A s \log x} \Big|^{\ell}  \\
& \leq e^{\sigma \epsilon} x^{\sigma}, 
\end{align*}
where the second inequality follows from an application of  \eqref{eqn:WeightChoice_ElementaryFunction} and the observation that $\tfrac{1}{2} + 2\ell A < \tfrac{1}{2} + \epsilon < 1$. For $s = \sigma > 0$, observe that $F(-\sigma\log x)$ is real and positive. Thus, by (iii) and \eqref{eqn:WeightChoice_ElementaryFunction},
\begin{align*}
F(-\sigma\log x) & 
\leq   e^{\sigma\epsilon} x^{\sigma}  \cdot \Big(\frac{1- x^{-(\tfrac{1}{2}+2\ell A) \sigma}}{\sigma \log x} \Big)  \cdot \Big( \frac{1-x^{-2A \sigma}}{2A \sigma \log x} \Big)^{\ell}  \\
& \leq   \frac{e^{\sigma\epsilon} x^{\sigma}}{\sigma \log x} \cdot \Big( \frac{1-x^{-2A \sigma }}{2A \sigma \log x} \Big)^{\ell}  \\
& \leq \frac{e^{\sigma \epsilon} x^{\sigma}}{\sigma \log x}.
\end{align*} 
This completes the proof of all cases of (iv).
	\item For part (vi), we shall argue as in (iv). Rearranging (iii), notice that
\[
|F(z)| = \Big|e^{(-\frac{1}{2}+2\ell A)z} \cdot \Big(\frac{1-e^{-(\frac{1}{2}+2\ell A)z}}{z} \Big) \Big(\frac{1-e^{-2Az}}{2Az}\Big)^{\ell}\Big|.
\]
If $r := \Re\{z\} > 0$, then 
\begin{align*}
|F(z)| 
& \leq e^{(-\frac{1}{2}+2\ell A)r} \cdot  \frac{1 + e^{-(\frac{1}{2}+2\ell A)r}}{|z|} \cdot  \Big(\frac{1+e^{-2Ar}}{2A|z|}\Big)^{\ell} \\
& \leq \frac{2 e^{(-\frac{1}{2}+2\ell A)r}}{|z|} \Big(\frac{1}{A|z|}\Big)^{\ell}.
\end{align*}
If we substitute $z = -s\log x = (\tfrac{1}{2}-it) \log x$, then it follows by the definition of $A$ that 
\begin{align*}
|F(-s\log x)| 
& \leq \frac{2 e^{\epsilon/2} x^{-1/4}}{|\tfrac{1}{2}+it| \log x} \Big(\frac{2\ell}{\epsilon |\tfrac{1}{2}+it|} \Big)^{\ell}
 \leq \frac{4 e^{\epsilon/2} x^{-1/4}}{\log x} \Big(\frac{2\ell}{\epsilon } \Big)^{\ell} (1/4+t^2)^{-\ell/2}.
\end{align*}
This yields (vi) since $4e^{\epsilon/2} < 5$ for $\epsilon < 1/4$. 

\end{itemize}

\end{proof}

\section{Preliminary Analysis}
\label{sec:FirstSteps}

\subsection{A weighted sum of prime ideals} 
\label{subsec:WeightPrimes}
For $x > 3, \epsilon \in (0,1/4)$ and integer $\ell \geq 1$, use the compactly-supported weight $f(\, \cdot \,)  = f(\, \cdot \, ; x, \ell, \epsilon)$ defined in \cref{lem:WeightChoice} and set
\begin{equation}
S(x) = S_{\ell,\epsilon}(x):=  \sum_{\kn \subseteq \cO_F} \Lambda_F(\kn) \Theta_C(\kn)	f\Big( \frac{\log \N\kn}{\log x} \Big).
\label{def:S_WeightedPrimes}
\end{equation}
We reduce our estimation of $\pi_C(x,L/F)$ given by \eqref{def:pi_C} to the smoothed version $S(x)$. 
%
\begin{lem} \label{lem:ReduceMainTheorems}
Let $x_0 > e^4$. Suppose there exist constants $a,b \geq 0$ and $0 \leq c \leq 1/2$, all of which are independent of $x$, such that $S(x) < \big\{ a + b x^{-c} \big\} \frac{|C|}{|G|} x$ for all $x \geq x_0$. Then, for all $x \geq x_0$,
\[
\pi_C(x,L/F) < \Big\{ a + 2b x^{-c} + O\Big( \frac{n_L}{x^{1/2}} + \frac{n_L x_0 \log x }{x} \Big) \Big\}  \frac{|C|}{|G|} \mathrm{Li}(x). 
\]
\end{lem}
\begin{proof}
	
	If $t > 1$, then
	\begin{equation}
	\label{eqn:eqn1}
	\psi_C(t) =  \sum_{\substack{ t^{1/2} \leq \N\kn < t}} \Theta_C(\kn) \Lambda_K(\kn) + \psi_C(t^{1/2}).
	\end{equation}
	The sum in \eqref{eqn:eqn1} is bounded by $S(t)$ in \eqref{def:S_WeightedPrimes} because of \cref{lem:WeightChoice}(i), while the secondary term in \eqref{eqn:eqn1} is estimated much like \eqref{eqn:Pi_Bound}. Thus, we have that
	\begin{equation}
	\label{eqn:eqn2}
	\psi_C(t)\leq S(t) + O( n_F t^{1/2} ).
	\end{equation}
	We substitute \eqref{eqn:eqn2} into \cref{lem:Pi_to_Psi} and deduce that
	\[
	\pi_C(x,L/F) \leq \frac{S(x)}{\log x} + \int_{x_0}^x \frac{S(t)}{t \log^2 t} dt  + O\Big( \frac{n_F x^{1/2}}{\log x} + n_F x_0\Big).
	\]
	From our assumption on $S(t)$ for $t \geq x_0$, it follows that
	\begin{equation}
	\label{eqn:ReduceTheorem_LMO_Sx}
	\pi_C(x,L/F) 
	  < a \frac{|C|}{|G|} \mathrm{Li}(x) + b\frac{|C|}{|G|} \Big[ \frac{ x^{1-c}}{\log x} + \int_{x_0}^x \frac{t^{-c} }{\log^2 t} dt  \Big] + O\Big( \frac{n_F x^{1/2}}{\log x} + n_F x_0\Big).  \\
	\end{equation}
	Note that if $0\leq c\leq 1/2$, then $t^{1-c}/\log^2 t$ is an increasing function of $t$ for $t > e^4$.  Since $x_0 > e^4$ and $\mathrm{Li}(x) > \frac{x}{\log x}$ for $x > e^4$, we conclude that
	\begin{equation}
	\label{eqn:t^-c}
	\int_{x_0}^x \frac{t^{-c}}{\log^2 t} dt = \int_{x_0}^x \frac{t^{1-c}}{\log^2 t} \frac{dt}{t} \leq \frac{x^{1-c}}{\log^2 x} \int_{x_0}^x \frac{dt}{t} \leq \frac{x^{1-c}}{\log x} < x^{-c} \,\mathrm{Li}(x).
	\end{equation}
	The desired result follows from \eqref{eqn:ReduceTheorem_LMO_Sx}, \eqref{eqn:t^-c}, and the identity $n_L = [L:F] n_F = |G| n_F$.
	\end{proof}
\subsection{Reduction to Hecke $L$-functions}  
\label{subsec:DeuringTrick}
By Mellin inversion, \eqref{def:S_WeightedPrimes}, and \eqref{eqn:ZC_orthogonality}, it follows that
\begin{equation}
S(x) = \frac{\log x}{2\pi i} \int_{2-i\infty}^{2+i\infty} Z_C(s) F(-s \log x) ds.
\label{eqn:S_MellinInversion}
\end{equation}
To shift the contour, we must rewrite $Z_C(s)$, defined by \eqref{def:ZC}, in terms of $L$-functions which exhibit an analytic continuation to the left of $\Re\{s\} = 1$.
To this end, let $H \subseteq G$ be an abelian subgroup such that $H \cap C$ is non-empty,  
and choose $g_C$ in \cref{subsec:PrimeCounters} so that $g_C \in H \cap C$.  Let $K = L^H$ be the subfield of $L$ fixed by $H$.  By standard arguments (see \cite[Theorem 3.7]{DW} and \cite[Section 3]{LMO}), we have that
\begin{equation}
Z_C(s) = - \frac{|C|}{|G|} \sum_{\chi\in\hat{H}} \bar{\chi}(g_C) \frac{L'}{L}(s,\chi,L/K), 
\label{eqn:ZC_classfieldtheory}
\end{equation}
where the sum runs over certain primitive Hecke characters $\chi$ of $K$ satisfying
\[
\chi(\kP) = \chi\Big( \Big[\frac{L/K}{\kP} \Big]\Big)
\]
for prime ideals $\kP$ of $K$ that are unramified in $L$.  Substituting \eqref{eqn:ZC_classfieldtheory} into \eqref{eqn:S_MellinInversion}, we conclude that
\begin{equation}
	S(x) = \frac{|C|}{|G|} \sum_{\chi} \bar{\chi}(g_C) \frac{\log x}{2\pi i} \int_{2 - i\infty}^{2+i\infty} - \frac{L'}{L}(s,\chi,L/K) F(-s\log x) ds.
	\label{eqn:SumToContour}
\end{equation}
Henceforth, any sum over $\chi$ is over all $\chi \in \hat{H}$. These are equivalently the Hecke characters attached to the abelian extension $L/K$ by class field theory. 

\subsection{Hecke $L$-functions}
\label{subsec:Hecke_L-fcns}
For a more detailed reference on Hecke $L$-functions, see \cite{LO} for example. Suppose $L/K$ is an abelian extension, so all irreducible representations of $\mathrm{Gal}(L/K)$ are 1-dimensional \emph{primitive} Hecke characters $\chi$ satisfying
\[
\chi(\kP) = \chi\Big( \Big[ \frac{L/K}{\kP} \Big] \Big)
\]
for prime ideals $\kP$ of $K$ that are unramified in $L$. The Hecke $L$-function of $\chi$ is defined by
\begin{equation}
L(s,\chi,L/K) = \sum_{\kN \subseteq \cO_K} \chi(\kN) \N\kN^{-s} = \prod_{\kP} \Big(1-\frac{\chi(\kP)}{\N\kP^{s} } \Big)^{-1} 
\label{def:Hecke_L-fcn}
\end{equation}
for $\Re\{s\} > 1$, where the sum is over integral ideals $\kN$ of $K$ and the product is over prime ideals $\kP$ of $K$. For this subsection only, we write $L(s,\chi) = L(s,\chi,L/K)$ and suppress the implicit dependence of quantities on the extension $L/K$. Define the completed Hecke $L$-function $\xi(s, \chi)$ by
\begin{equation}
\xi(s, \chi) = ( s(s-1))^{\delta(\chi)} D_{\chi}^{s/2} \gamma_{\chi}(s) L(s, \chi),
\label{def:CompletedHecke}
\end{equation}
where $D_{\chi} = D_K \N\kf_{\chi}$, the $K$-integral ideal $\kf_{\chi}$ is the conductor of $\chi$, $\delta(\chi)$ is the indicator function for the trivial character, and $\gamma_{\chi}(s)$ is the \emph{gamma factor of $\chi$} defined by
\begin{equation*}
\gamma_{\chi}(s) =  \Big[ \pi^{-\tfrac{s}{2}} \Gamma\Big(\frac{s}{2}\Big)  \Big]^{a(\chi)} \cdot \Big[ \pi^{-\tfrac{s+1}{2} } \Gamma\Big( \frac{s+1}{2} \Big)   \Big]^{b(\chi)}.
\label{eqn:GammaFactor} 
\end{equation*}
Here $a(\chi)$ and $b(\chi)$ are certain non-negative integers satisfying 
\begin{equation}
a(\chi) + b(\chi) = n_K. 
\label{eqn:GammaFactor_Exponents}
\end{equation}
It is well-known that $\xi(s, \chi)$ is entire of order 1 and satisfies the functional equation
\begin{equation*}
\xi(s, \chi) = w(\chi) \xi(1-s, \bar{\chi}),
\label{eqn:FunctionalEquation}
\end{equation*}
where $w(\chi) \in \mathbb{C}$ is the root number of $\chi$ satisfying $|w(\chi)| = 1$. 
The zeros of $\xi(s,\chi)$ are the non-trivial zeros $\rho$ of $L(s,\chi)$ and are known to satisfy $0 < \Re\{\rho\} < 1$.  The trivial zeros $\omega$ of $L(s, \chi)$ are given by 
\begin{equation}
\mathop{\ord}_{s = \omega} L(s, \chi) = 
\begin{cases}
a(\chi) - \delta(\chi) & \text{if } \omega = 0, \\
b(\chi) &  \text{if } \omega = -1,-3,-5,\dots,\\
a(\chi) & \text{if } \omega = -2,-4,-6, \dots,
\end{cases}
\label{eqn:TrivialZeros}
\end{equation}
and arise as poles of the gamma factor of $L(s,\chi)$.

\subsection{Shifting a contour integral} 
\label{subsec:ShiftContour}
Next we shift the contour \eqref{eqn:SumToContour} and bound $S(x)$ in terms of the non-trivial zeros of Hecke $L$-functions. Henceforth write $S = S(x)$ for simplicity. Recall $f$ depends on the arbitrary quantities $x > 3, \epsilon \in (0,1/4)$ and an integer $\ell \geq 1$.  

\begin{lem} 
\label{lem:ShiftedContour_Dedekind} Assume $\ell \geq 2$. Then
\begin{equation}	
	\begin{aligned}
	\frac{|G|}{|C|} \, \frac{S}{e^{\epsilon} x}   \leq 1 +   \frac{\log x}{e^{\epsilon} x} \sum_{\chi} \sum_{\rho_{\chi}} |F( -\rho_{\chi} \log x)|   + O\Big(n_L  x^{-1} \log x + x^{-5/4 } (2\ell /\epsilon)^{\ell} \log D_L \Big),
		\end{aligned}
	\label{eqn:ShiftedContour_Dedekind}
\end{equation}
where the outer sum is over all Hecke characters $\chi$ of the abelian extension $L/K$ and the inner sum runs over all non-trivial zeros $\rho_{\chi}$ of $L(s,\chi,L/K)$, counted with multiplicity. 
\end{lem}
\begin{proof}
	Shift the contour in \eqref{eqn:SumToContour} to the line $\Re\{s\} = -\tfrac{1}{2}$. This picks up  the non-trivial zeros of $L(s,\chi)$, the simple pole at $s=1$ when $\chi$ is trivial, and the trivial zero at $s=0$ of $L(s,\chi)$ of order  $r(\chi)$. Overall, we see that
	\begin{equation}	
	\begin{aligned}
	\frac{|G|}{|C|}S&  = \log x \Big[ F(-\log x) - \sum_{ \substack{ \chi  }} \bar{\chi}(g_C) \sum_{\rho_{\chi}} F( -\rho_{\chi} \log x)  + O\Big(\sum_{\chi} r(\chi) |F(0)| \Big) \Big]  \\
	& + \log x\sum_{\chi }  \frac{\bar{\chi}(g_C)}{2\pi i} \int_{-1/2-i\infty}^{-1/2+i\infty} -\frac{L'}{L}(s,\chi, L/K) F(-s \log x) ds,
	\end{aligned}
	\label{eqn:I_ShiftContour_Dedekind}
	\end{equation}
		where the sum over $\rho = \rho_{\chi}$ is over all non-trivial zeros of $L(s,\chi,L/K)$, counted with multiplicity. From \eqref{eqn:GammaFactor_Exponents} and \eqref{eqn:TrivialZeros}, we see $r(\chi) \leq n_K$; hence, it follows  by \cref{lem:WeightChoice}(v) that
		\[
		F(-\log x) \leq \frac{e^{\epsilon} x}{\log x}, \quad \text{and} \quad \sum_{\chi} r(\chi)|F(0)| \leq  [L:K] n_K = n_L. 
		\]
		For the remaining contour, by \cite[Lemma 6.2]{LO} and the primitivity of $\chi$, we have that
	\begin{align*}
	-\frac{L'}{L}(s,\chi, L/K) & \ll \log D_{\chi} + n_K \log(|s|+3), 
	\end{align*}
	for $\Re\{s\} = -1/2$ and where $D_{\chi}$ is defined in \eqref{def:CompletedHecke}. It follows by \cref{lem:WeightChoice}(vi) that
	\begin{align*}
	& \frac{\log x}{2\pi i} \int_{-1/2-i\infty}^{-1/2+i\infty}-\frac{L'}{L}(s,\chi, L/K) F(-s \log x) ds  \\
	&\ll x^{-1/4}\Big( \frac{2\ell}{\epsilon} \Big)^{\ell}  \int_{-\infty}^{\infty} \frac{\log D_{\chi} + n_K \log(|t|+3)}{(1/4+t^2)^{\ell/2}} dt  \ll x^{-1/4 } \Big( \frac{2\ell}{\epsilon} \Big)^{\ell}   \log D_{\chi},
	\end{align*}
	because $n_K \ll \log D_K \leq \log D_{\chi}$ and $\ell \geq 2$. Summing over $\chi$ and using the conductor-discriminant formula yields
	\[
	\log x\sum_{\chi }  \frac{\bar{\chi}(g_C)}{2\pi i} \int_{-1/2-i\infty}^{-1/2+i\infty} -\frac{L'}{L}(s,\chi, L/K) F(-s \log x) ds
	 \ll x^{-1/4} \Big( \frac{2\ell}{\epsilon} \Big)^{\ell} \log D_L.
	\]
	 Taking absolute value of both sides in \eqref{eqn:I_ShiftContour_Dedekind}, multiplying both sides by $(e^{\epsilon} x)^{-1}$, and combining all of these observations yields the desired result.
\end{proof}

To analyze the sum over zeros in \cref{lem:ShiftedContour_Dedekind}, we require some information about the distribution of zeros of Hecke $L$-functions.

\section{Distribution of Zeros of Hecke $L$-functions}
\label{sec:DistributionZerosHecke}

In this section, we record various results about $L$-functions $L(s,\chi,L/K)$ where the extension $L/K$ is abelian and hence $\chi$ is a Hecke character of $K$ by class field theory. Associated notation and classical results can be found in \cref{sec:Preliminaries}. Henceforth, any sum $\sum_{\chi}$ or product $\prod_{\chi}$ is over all characters $\chi$ of $L/K$ unless otherwise specified. 

\subsection{Logarithmic Quantity} Let $\delta_0 > 0$ be fixed and sufficiently small. For the remainder of the paper, denote
\begin{equation}	
\sL := \begin{cases}
(\tfrac{1}{3}+\delta_0) \log D_K + (\tfrac{19}{36}+\delta_0) \log \cQ + (\tfrac{5}{12} + \delta_0)   n_K \log n_K & \text{if $n_K^{\frac{5n_K}{6}} \geq D_K^{\frac{4}{3}} \cQ^{\frac{4}{9}},$} \\
(1+\delta_0) \log D_K + (\tfrac{3}{4}+\delta_0) \log \cQ +  \delta_0 n_K \log n_K & \text{otherwise,}	
 \end{cases}
 \label{def:sL}
\end{equation}
where $\cQ = \cQ(L/K) = \max\{ \N \kf_{\chi} : \chi \in \widehat{\mathrm{Gal}} (L/K) \}$. Notice that
\begin{equation}
\sL  \geq (1+\delta_0)\log D_K + (\tfrac{3}{4}+\delta_0) \log \cQ + \delta_0 n_K\log n_K \quad \text{and} \quad \sL \geq (\tfrac{5}{12} + \delta_0) n_K \log n_K
\label{eqn:sL_lowerbound}
\end{equation}
unconditionally. We exhibit a bound on the degree of the extension $L/K$ in terms of $\sL$. 

\begin{lem}
	\label{lem:CharacterGroupBound}
	$[L:K] \ll e^{4\sL/3}$	and $n_L \ll \sL e^{4\sL/3}$.
\end{lem}
\begin{proof}
	Let $\kf = \kf_{L/K}$ be the Artin conductor attached to $L/K$ by class field theory. Let $I(\kf)$ be the group of fractional ideals of $K$ relatively prime to $\kf$.  By class field theory, there exists a homomorphism $\phi:I(\kf)\to \mathrm{Gal}(L/K)$.  Thus $I(\kf)/\ker\phi$ is isomorphic to $\mathrm{Gal}(L/K)$. This induces an isomorphism between their respective character groups and therefore,
	\[
	\cQ(L/K) = \max\{ \N \kf_{\chi} : \chi \in \widehat{\mathrm{Gal}} (L/K) \} = \max\{ \N \kf_{\chi} : \chi \in \widehat{I(\kf)/\ker\phi} \}.
	\]
	By our previous observations, $|I(\kf)/\ker\phi| = |\mathrm{Gal}(L/K)| = [L:K]$. For $\epsilon_0 > 0$ fixed and sufficiently small, we have  by \cite[Lemma 2.11]{JT_AZ} that $|I(\kf)/\ker\phi| \ll e^{O_{\epsilon_0}(n_K)} D_K^{1/2+\epsilon_0} \cQ^{1+\epsilon_0} \ll e^{4\sL/3}$ as desired. To bound $n_L$, observe that $n_L = [L:K] n_K$ and $n_K \ll \sL$. 
\end{proof}

\subsection{Low-Lying Zeros}
\label{subsec:LPI-notation}

Next, we specify some important zeros of $\prod_{\chi} L(s,\chi,L/K)$ which will be used in \cref{sec:Proof_BT_easy,sec:Proof_BT_SiegelZeroExists,sec:Proof_BT_NoSiegelZero}. For the remainder of the paper, let $\eta > 0$ be sufficiently small and arbitrary. Consider the multiset of zeros given by
\begin{equation}
\cZ := \Big\{ \rho \in \mathbb{C} : \prod_{\chi} L(\rho,\chi,L/K) = 0, 0 < \Re\{ \rho\} < 1, |\Im(\rho)| \leq \eta^{-2} \Big\}. 
\label{eqn:SelectZeros}
\end{equation}	
We select three important zeros of $\mathcal{Z}$ as follows:  
\begin{itemize}
	\item Choose $\rho_1 \in \cZ$ such that $\Re\{\rho_1\}$ is maximal. Let $\chi_1$ be its associated Hecke character so $L(\rho_1,\chi_1,L/K) = 0$. Denote
\[
\rho_1 = \beta_1 + i\gamma_1 = \Big(1-\frac{\lambda_1}{\sL}\Big) + i \frac{\mu_1}{\sL},
\]
where $\beta_1 = \Re\{\rho_1\}, \gamma_1 = \Im\{\rho_1\}, \lambda_1 > 0,$ and $\mu_1 \in \R$. 
	\item Choose\footnote{If $\rho_1$ is real then $\rho' \in \cZ \setminus \{\rho_1\}$ instead, with the other conditions remaining the same.} $\rho' \in \cZ \setminus\{ \rho_1, \bar{\rho_1}\}$ satisfying  $L(\rho',\chi_1,L/K) = 0$ such that $\Re\{\rho'\}$ is maximal with respect to these conditions. Similarly denote
	\[
	 \rho' = \beta' + i\gamma'=\Big(1-\frac{\lambda'}{\sL}\Big) + i \frac{\mu'}{\sL}.
	\]
	\item Choose $\rho_2 \in \cZ \setminus \cZ_1$ such that $\Re\{\rho_2\}$ is maximal and where $\cZ_1$ is the multiset of zeros of $L(s,\chi_1,L/K)$ contained in $\cZ$. Let $\chi_2$ be its associated Hecke character so $L(\rho_2,\chi_2,L/K) = 0$. Similarly, denote
\[
\rho_2 = \beta_2 + i\gamma_2=\Big(1-\frac{\lambda_2}{\sL}\Big) + i \frac{\mu_2}{\sL}.
\]
\end{itemize}

If $\lambda_1 < \eta$ then we henceforth refer to $\rho_1$ as an \emph{$\eta$-Siegel zero}. The proof of \cref{thm:BT_Weiss_sharp} will be divided according to whether an $\eta$-Siegel zero exists or not.

\subsection{Zero-Free Regions} 
Here we record the current best-known explicit result regarding zero-free regions  of Hecke $L$-functions; see also \cite{Ahn-Kwon, Kadiri} for earlier results. 

\begin{thm}[Zaman] \label{thm:ZeroFreeRegions}
For $\sL$ sufficiently large depending on $\eta$, $\min\{\lambda',\lambda_2\} > 0.2866$.  Furthermore, if $\lambda_1 < 0.0875$ then $\rho_1$ is a simple real zero of $\prod_{\chi} L(s,\chi, L/K)$ and is associated with a real character $\chi_1$. 
\end{thm}
\begin{proof} 
	When $L$ is a narrow ray class field of $K$ to a given modulus and $\eta = 1$ in \eqref{eqn:SelectZeros}, this is implied by \cite[Theorems 1.1 and 1.3]{Zaman_2015a} since $\sL$ satisfies \eqref{eqn:sL_lowerbound}. For general abelian extensions $L/K$ and any fixed $\eta \in (0,1)$, one may easily modify \cite{Zaman_2015a} to obtain the cited result by following the outline in \cite[Section 8]{JT_AZ}; see \cite{Zaman_thesis} for details. 
\end{proof}

\subsection{Zero Repulsion}

Here we record two explicit estimates for zero repulsion when an exceptional zero exists, also known as ``Deuring-Heilbronn phenomenon". 

\begin{thm}[Zaman]
\label{thm:ZeroRepulsion} Let $\sL$ be sufficiently large depending on $\eta$. 
If $\lambda_1 < 0.0875$, then $\min\{\lambda', \lambda_2\} > 0.44$.  If $\eta \leq \lambda_1 < 0.0875$, then $\min\{\lambda', \lambda_2\} > 0.2103 \log(1/\lambda_1)$.
\end{thm}

\begin{proof} Again, when $L$ is a narrow ray class field of $K$ to a given modulus and $\eta = 1$,  this is implied by \cite[Theorem 1.4]{Zaman_2015a} since $\sL$ satisfies \eqref{eqn:sL_lowerbound}. Similar to the proof of \cref{thm:ZeroFreeRegions}, one may modify \cite{Zaman_2015a} as outlined in \cite[Section 7]{JT_AZ} to deduce the same theorem for general abelian extensions $L/K$ and $\eta \in (0,1)$; see \cite{Zaman_thesis} for details.
\end{proof}

\cref{thm:ZeroRepulsion} is unable to handle exceptional zeros $\rho_1$ extremely close to 1 due to the requirement $\lambda_1 \geq \eta$. Thus, we include a version of Deuring-Heilbronn phenomenon \cite[Theorem 8.3]{JT_AZ} which repels zeros in the entire critical strip. 

\begin{thm}[Thorner--Zaman]
\label{thm:ZeroRepulsion_Full}
Let $T \geq 1$ be arbitrary. Suppose $\chi_1$ is a real character and $\rho_1$ is a real zero. For any character $\chi$ of $L/K$, let $\rho  = \beta + i\gamma \neq \rho_1$ be a non-trivial zero of $L(s,\chi,L/K)$ satisfying $1/2 \leq \beta < 1$ and $ |\gamma| \leq T$.  For $\sL$ sufficiently large, there exists an absolute effectively computable constant $c_1 > 0$ such that
\[
\beta < 1- \frac{\log\Big(\dfrac{c_1}{(1-\beta_1)(\sL + n_K \log T)}\Big)}{81\sL + 25 n_K \log T}.
\]
\end{thm}

\subsection{Log-Free Zero Density Estimates} 
\label{subsec:DistributionZerosHecke_LFZD}
Let $\chi \in \widehat{\mathrm{Gal}}(L/K)$ be a Hecke character. Define
\[
N(\sigma,T,\chi) := \#\{ \rho = \beta+i\gamma : L(\rho,\chi,L/K) = 0, \sigma < \beta < 1, |\gamma| \leq T \}
\]
for $0 < \sigma < 1$ and $T \geq 1$. Further denote
\begin{equation}
	N(\sigma,T) := \sum_{\chi} N(\sigma,T,\chi). 
	\label{def:NumberOfZeros}
\end{equation}
Amongst all of the results recorded herein on zeros of Hecke $L$-functions, the proof of \cref{thm:BT_Weiss} only requires the following log-free zero density estimate, which we emphasize does \emph{not} assume $\sL$ is sufficiently large. This is a rephrasing of the authors' result \cite[Theorem 3.2]{JT_AZ} using the definition of $\sL$ in \eqref{def:sL}.

\begin{thm}[Thorner--Zaman] \label{thm:LFZD_HighLying}
	For $0 < \sigma < 1$ and $T \geq 1$, $N(\sigma,T) \ll (e^{162\sL} T^{81n_K+162})^{1-\sigma}$.
\end{thm}

The proof of \cref{thm:BT_Weiss_sharp} also requires a completely explicit zero density estimate for ``low-lying" zeros. Define for $0 < \lambda < \sL$,
\begin{equation}
\begin{aligned}
\cN(\lambda) & := \sum_{\chi} N(1-\tfrac{\lambda}{\sL}, \eta^{-2}, \chi).
\end{aligned}
\label{def:N_lambda}
\end{equation}
\cref{thm:ZeroFreeRegions} states that $\cN(0.0875) \leq 1$ and $\cN(0.2866) \leq 2$ for $\sL$ sufficiently large depending on $\eta$. For larger values of $\lambda$, we use the following:

\begin{thm}[Thorner--Zaman] \label{thm:LFZD-LowLying} Assume $\sL$ is sufficiently large depending on $\eta$.
Let $\epsilon_0 > 0$ be fixed and sufficiently small. If $0 < \lambda <  \epsilon_0 \sL$ then 
\[
\cN(\lambda) \leq e^{162 \lambda + 188}.
\]
The bounds for $\cN(\lambda)$ in \cite[Table 1]{JT_AZ} are superior when $0 < \lambda \leq 1$. 
\end{thm}
\begin{proof}
See \cite[Theorem 8.6]{JT_AZ} for details.	
\end{proof}

\section{Zeros outside a low-lying rectangle}
\label{sec:NegligibleZeros}

From \cref{lem:ShiftedContour_Dedekind}, it remains to estimate a sum over all non-trivial zeros of all Hecke $L$-functions $L(s,\chi,L/K)$. In this section, we demonstrate that the contribution of zeros is negligible if the zeros are either high-lying or far from the line $\Re\{s\}=1$. Throughout, we assume $1 \leq B \leq 1000$ is a fixed absolute constant. We begin by considering high-lying zeros. 

\begin{lem}
\label{lem:HighLyingZeros}
Let $T_{\star} \geq 1$ be arbitrary. Let $0 < E < \tfrac{2}{3}B$ be fixed.  Let 
\begin{equation}
B > 162 + E, \quad \ell \geq  82 n_K + 162, \quad \tfrac{1}{4} > \epsilon \geq 4 \ell x^{-E/(B\ell)}.
		\label{eqn:HighLyingZeros_assumptions}
		\end{equation}
For  $x \geq e^{B \sL}$, 
\begin{equation}
\frac{\log x}{x} \sum_{\chi} \sum_{\substack{\rho \\ |\Im\{\rho\}| >  T_{\star}} } |F(- \rho \log x)| \ll \frac{1}{T_{\star}}.
\label{eqn:HighLyingZeros}
\end{equation}
\end{lem}

\begin{proof} 
	Write $\rho = \beta+i\gamma$ with $\beta = 1-\frac{\lambda}{\sL}$. If $T \geq 1$, then \cref{lem:WeightChoice}(iv) with $\alpha = \ell(1-\beta)$ and our choices of our conditions on $\epsilon, \ell,$ and $x$ imply that  
	\begin{equation}
		\begin{aligned}
		\frac{\log x}{x} |F(-\rho \log x)| 
 	\leq \frac{2e^{\epsilon} x^{\beta-1}}{T}   \Big( \frac{2\ell}{\epsilon T} \Big)^{\ell(1-\beta)}  	\leq \frac{4}{T} e^{-(B- E) \lambda} (2T)^{-(82n_K + 162)\lambda/\sL}.
		\end{aligned}
	\end{equation}
Using \cref{thm:LFZD_HighLying} via partial summation, we see that
	\begin{align*}
	& \frac{T\log x}{x} \sum_{\chi} \sum_{\substack{\rho \\ T \leq |\Im\{\rho\}| \leq 2T} } |F(-\rho \log x)|  \\
	& \quad \ll \frac{e^{-(B-E-162)\sL}}{ (2T)^{n_K}} + \Big(B-E + \frac{n_K \log(2T)}{\sL}\Big)\int_{0}^{\sL} e^{-(B-E-162)\lambda} (2T)^{-n_K\lambda/\sL} d\lambda\ll 1,
	\end{align*}
	since $B > 162 + E$. Overall, this implies that the LHS of \eqref{eqn:HighLyingZeros} is
	\begin{equation*}
	\begin{aligned}
		 & \leq \frac{\log x}{x} \sum_{\chi} \sum_{k=0}^{\infty} \sum_{ \substack{\rho \\ 2^k T_{\star} \leq \Im\{\rho\} < 2^{k+1} T_{\star} }} |F(-\rho \log x)|  
		 \ll \frac{1}{T_{\star}}  \sum_{k=0}^{\infty} \frac{1}{2^k}
		 \ll \frac{1}{T_{\star}}, 
	\end{aligned}	
	\end{equation*}
	as desired. 
\end{proof}

As we shall see in the next section, an appropriate combination of \cref{lem:ShiftedContour_Dedekind,lem:HighLyingZeros,thm:LFZD_HighLying} suffices to establish \cref{thm:BT_Weiss}. For \cref{thm:BT_Weiss_sharp}, we must also show low-lying zeros far to the left of $\Re\{s\}=1$ contribute a negligible amount.   

\begin{lem}
	\label{lem:LowLyingZeros} Let $0 \leq R \leq \tfrac{1}{2}\sL$ be arbitrary.	Assume \eqref{eqn:HighLyingZeros_assumptions} holds. For $x \geq e^{B\sL}$, 
	\begin{equation*}
	\begin{aligned}
	\frac{\log x}{x} \sum_{\chi}  \sumP_{\rho} |F(-\rho \log x)|
	& \ll x^{-(B - E - 162) R/B\sL}
	\end{aligned}
\end{equation*}
where the marked sum $\sum'$ runs over zeros $\rho = \beta+i\gamma$ of $L(s,\chi,L/K)$, counting with multiplicity, satisfying $0 <\beta \leq 1 - R/\sL$ and $|\gamma| \leq \epsilon^{-1}$.
\end{lem}
\begin{proof} From our choices of $\epsilon, \ell$ in \eqref{eqn:HighLyingZeros_assumptions} and \cref{thm:LFZD_HighLying},  it follows that
	\[
	N(1-\tfrac{\lambda}{\sL}, \epsilon^{-1}) \ll e^{162 \lambda} (1/\epsilon)^{(81n_K+162)\lambda/\sL} \ll e^{162 \lambda} x^{E\lambda/B\sL} \ll x^{(162+E)\lambda/B\sL}
	\]
	for $0 < \lambda < \sL$, where $N(\sigma,T)$ is given by \eqref{def:NumberOfZeros}. Write $\rho = \beta+i\gamma$ with $\beta = 1-\frac{\lambda}{\sL}$ for some non-trivial zero $\rho$ appearing in the marked sum. By \cref{lem:WeightChoice}(iv) with $\alpha = 0$ and \cref{lem:WeightChoice}(v), it follows that
	\begin{equation}
	\frac{\log x}{x} |F(-\rho \log x)| \ll \begin{cases} x^{-\lambda/\sL} & \text{for $|\rho| \geq 1/4$}, \\ 
 		x^{-3/4} \log x & \text{for $|\rho| \leq 1/4$.}	
	 \end{cases}
	 \label{eqn:LowLyingZeros_Contribution}
	\end{equation}
	To clarify the second inequality, we observe by \cref{lem:WeightChoice}(v) that $|F(-\rho \log x)| \ll x^{\beta} \ll x^{1/4}$ for $|\rho| \leq 1/4$. Thus, by \eqref{eqn:LowLyingZeros_Contribution} and partial summation, we have that 
	\begin{align*}
	 \frac{\log x}{x}\sum_{\chi} \sumP_{\substack{ |\rho| \geq 1/4}} |F(-\rho\log x)| 
	&  \ll x^{\frac{-(B-E-162)}{B}} + \frac{\log x}{\sL}\int_{R}^{\sL} x^{\frac{-(B-E-162)\lambda}{B\sL}} d\lambda \\  
	& \ll x^{-(B - E - 162) R/B\sL}.
	\end{align*}
	Moreover, by \eqref{eqn:LowLyingZeros_Contribution}, a crude application of \cite[Lemma 2.1]{LMO}, and \cref{lem:CharacterGroupBound}, it follows that
	\begin{align*}
	 \frac{\log x}{x}\sum_{\chi} \sumP_{\substack{ \rho \\ |\rho| \leq 1/4}} |F(-\rho\log x)| 
	&   \ll [L:K] \sL  x^{-3/4} \log x \ll x^{-3/4} e^{2 \sL} \log x \ll x^{-\frac{3}{4}+\frac{3}{B}}.
	\end{align*}
	Combining these estimates yields the desired result since, by our assumptions on $B$ and $R$, $x^{-(B - E - 162) R/B\sL} \gg x^{-(B-E-162)/2B} \gg x^{-1/2} \gg x^{-3/4 + 3/162} \gg x^{-3/4+3/B}$.  
	\end{proof}
We package these lemmas into the following convenient proposition. 

\begin{prop} \label{prop:FinalArgument}
	Let $0 \leq R \leq \tfrac{1}{2}\sL$ be arbitrary. Let $0 < E < \tfrac{2}{3}B$ be fixed. Assume that
	\begin{equation}
		B >  162+E,
		\qquad
		\ell \geq 82n_K + 162,
		\qquad
		\tfrac{1}{4} > \epsilon \geq 4\ell x^{-E/(B\ell)}.
		\label{eqn:FinalArgument_assumptions}
	\end{equation}
	If $x \geq e^{B \sL}$ and $S(x)$ is given by \eqref{def:S_WeightedPrimes}, then
	\begin{equation}
		\begin{aligned}
		 \frac{|G|}{|C|} \frac{S(x)}{e^{\epsilon} x} 
			& \leq 1 + \frac{\log x}{e^{\epsilon}x} \sum_{\chi} \sumS_{\rho} |F(-\rho \log x)| 	+ O\big(\epsilon + x^{-(B - E - 162) R/B\sL} \big),
		\end{aligned}
		\label{eqn:FinalArgument_conclusion}
	\end{equation}
	where the sum $\sum^{\star}$ indicates a restriction to non-trivial zeros $\rho$ of $L(s,\chi,L/K)$, counted with multiplicity, satisfying $1-R/\sL < \Re\{\rho\} < 1$ and $|\Im\{\rho\}| \leq \epsilon^{-1}$.
\end{prop}

\begin{proof} 
	 Let $T_{\star} = 1/\epsilon$.  It follows from our hypothesis \eqref{eqn:FinalArgument_assumptions} along with \cref{lem:ShiftedContour_Dedekind,lem:HighLyingZeros,lem:LowLyingZeros} that
		\begin{equation}
		\begin{aligned}
		 \frac{|G|}{|C|} \frac{S}{e^{\epsilon} x}
			& \leq 1 + \frac{\log x}{e^{\epsilon}x} \sum_{\chi} \sumS_{\rho} |F(-\rho \log x)| 	\\
			& \qquad  + O\Big(\epsilon + x^{-(B - E - 162) R/B\sL} + n_L  x^{-1} \log x + x^{-5/4 } (2\ell /\epsilon)^{\ell} \log D_L \Big).
		\end{aligned}
		\label{eqn:FinalArguments_step1}
	\end{equation}
	It remains to bound the third and fourth expressions in the error term by $\epsilon$. Since $E < B$ and $\ell \geq 244$, we see that
	\[
	\epsilon > x^{-E/B \ell} > x^{-1/\ell}  > x^{-1/244}.
	\]
	Moreover, $n_L = n_K [L:K] \ll \sL e^{2\sL} \ll x^{3/162}$ by \cref{lem:CharacterGroupBound} and \eqref{eqn:sL_lowerbound}. Similarly, since $\log D_L = \sum_{\chi} \log D_{\chi} \leq [L:K] \log(D_K \cQ)$, it follows that
	\[
		(2\ell/\epsilon)^{\ell} \log D_L \ll x^{E/B}  \sL [L:K]  \ll   x^{2/3} \sL e^{2\sL} \ll x^{2/3+3/162}.
	\]
	Applying these estimates in \eqref{eqn:FinalArguments_step1} yields \eqref{eqn:FinalArgument_conclusion}. 
\end{proof}

\section{Proof of \cref{thm:BT_Weiss}}
\label{sec:Proof_BT_easy}

In comparison to \cref{thm:BT_Weiss_sharp}, the proof of \cref{thm:BT_Weiss} is quite simple, requiring only the log-free zero density estimate of Hecke $L$-functions given by \cref{thm:LFZD_HighLying}. Recall this result is uniform over all extensions $L/F$ and therefore we do not assume $\sL$ is sufficiently large. \\
\begin{proof}[Proof of Theorem \ref{thm:BT_Weiss}]
Select 
\begin{equation}
	B = 244.5, \quad  E = 82.1, \quad \ell = 82n_K + 162, \quad \epsilon = 1/8, \quad \text{and } R = 0.
	\label{eqn:Proof_Thm1.1_choices}
\end{equation}
Let $M_0 > 0$ be a sufficiently large absolute constant. For $x \geq x_0 := e^{244.5 \sL} + M_0 n_K^{244.5 n_K},$ we claim these are valid choices to invoke \cref{prop:FinalArgument}. It suffices to check $\epsilon = \tfrac{1}{8} \geq 4\ell x^{-E/B\ell}$ for $x \geq x_0$. We need only show $(32 \ell)^{B\ell/E} \leq x_0$. This is visible from the fact that
\[
(32\ell)^{B\ell/E} \ll n_K^{\frac{244.5}{82.1}(82n_K + 162)} e^{O(n_K)} \ll n_K^{244.5n_K} \leq x_0,
\]
after enlarging $M_0$ if necessary. This proves the claim.

Therefore, by \cref{prop:FinalArgument}, we have that $S(x) \ll \frac{|C|}{|G|} x$ for $x \geq x_0$, because the corresponding restricted sum $\sum^{\star}$ is empty whenever $R=0$. Let $M \geq 1$ denote the implicit absolute constant in the above estimate for $S(x)$. Thus, by \cref{lem:ReduceMainTheorems} with $x_0 = e^{244.5\sL} + M_0 n_K^{244.5 n_K}, a = M$ and $b=c=0$, we have that
\[
\pi_C(x,L/F) <  \Big\{ M +  O\big( n_L x^{-1/2} + \frac{n_L \log x}{x}(e^{244.5\sL} +   n_K^{244.5 n_K}) \big) \Big\} \frac{|C|}{|G|} \mathrm{Li}(x)
\]
for $x \geq x_0$. By \cref{lem:CharacterGroupBound} and \eqref{def:sL}, notice that $n_L \ll e^{4\sL/3} \ll D_K^2 \cQ^2 n_K^{n_K}$. Thus, the desired result follows for $x \gg e^{245.9 \sL} + D_K^2 \cQ^2 n_K^{246 n_K}$.  
\end{proof}

\begin{remark} 
	~
	\begin{itemize}
		\item 	
	If one wishes to minimize the value of $B$ and hence minimize the exponents of $D_K$ and $\cQ$ in \eqref{eqn:range_BT_Weiss} then one may alternatively select
	\[
	B = 162.01, \quad  E = 0.95, \quad \ell = 82n_K + 162, \quad \epsilon = 1/8, \quad \text{and } R = 0 
	\]	
	in place of \eqref{eqn:Proof_Thm1.1_choices} taking also $x_0 = e^{162.01\sL} + M_0 n_K^{13,999 n_K}$. It follows that
	\[
	(32\ell)^{B\ell/E} \ll n_K^{\frac{162.01}{0.95}(82n_K + 162)} e^{O(n_K)} \ll n_K^{13,999n_K} \leq x_0. 
	\]
	Arguing as above, one deduces $\pi_C(x, L/F) \ll \frac{|C|}{|G|} \mathrm{Li}(x)$ for $x \gg e^{164 \sL} +  D_K^ 2 \cQ^2 n_K^{14,000 n_K}$ as claimed in the remark following \cref{thm:BT_Weiss} based on \eqref{def:sL}. 

	\item Similarly, to minimize the exponents of $n_K^{n_K}$ in \eqref{eqn:range_BT_Weiss}, one may alternatively select
	\[
	B = 359.5, \quad  E = 197, \quad \ell = 82n_K + 162, \quad \epsilon = 1/8, \quad \text{and } R = 0
	\]	
	in place of \eqref{eqn:Proof_Thm1.1_choices} taking also $x_0 = e^{359.5\sL}$. It follows by \eqref{eqn:sL_lowerbound} that
	\[
	(32\ell)^{B\ell/E} \ll n_K^{\frac{359.5}{197}(82n_K + 162)} e^{O(n_K)} \ll n_K^{149.65 n_K} \leq x_0,
	\]
	since $359.5 \times \frac{5}{12} > 149.7$. Arguing as above, one deduces $\pi_C(x, L/F) \ll \frac{|C|}{|G|} \mathrm{Li}(x)$ for $x \gg e^{360.9 \sL} \geq e^{4\sL/3} e^{359.5\sL}$ as claimed in the remark following \cref{thm:BT_Weiss}. 
	
	\end{itemize}
\end{remark}

The following two sections consists of the proof of \cref{thm:BT_Weiss_sharp} which is divided into cases depending on how close the zero $\rho_1$, defined in \cref{subsec:LPI-notation}, is to $\Re\{s\}=1$. The main steps are similar to the above proof for \cref{thm:BT_Weiss} but need a more refined analysis.  

\section{Proof of \cref{thm:BT_Weiss_sharp}: $\eta$-Siegel zero exists}
\label{sec:Proof_BT_SiegelZeroExists}
Let $\eta > 0$ be arbitrary and sufficiently small and let $\sL$ be sufficiently large depending only on $\eta$. The proof of \cref{thm:BT_Weiss_sharp} is divided into \cref{sec:Proof_BT_SiegelZeroExists,sec:Proof_BT_NoSiegelZero} by whether $\rho_1$ is an $\eta$-Siegel zero or not. \\

For this section, we consider the case when $\lambda_1 < \eta$. By \cref{thm:ZeroFreeRegions}, it follows that $\rho_1 = \beta_1 = 1 - \frac{\lambda_1}{\sL}$ is a simple real zero and $\chi_1$ is a real Hecke character. Suppose
\begin{equation}
B = 692, \qquad E = 344, \qquad \ell = 82n_K + 162, \qquad 4\ell x^{-344/692 \ell} \leq \epsilon < 1/4.
\label{eqn:Exceptional_parameters}
\end{equation}
With these choices, we claim for $x \geq e^{692 \sL}$ that $4\ell x^{-344/692 \ell} = o(1)$ as $\sL \rightarrow \infty$. If $n_K$ is uniformly bounded while $\sL \rightarrow \infty$ then this is immediate, so we may assume $n_K \rightarrow \infty$.  By \eqref{eqn:sL_lowerbound}, notice that $\ell = 82n_K + 162 \leq \{ 196.8 + o(1)\} \frac{\sL}{\log n_K} \leq 197 \frac{\sL}{\log n_K}$ for $n_K$ sufficiently large. Thus, for $n_K$ sufficiently large and $x \geq e^{692 \sL}$, we have that
\[
4\ell x^{-344/692\ell} \ll n_K e^{-344 \sL/\ell} \ll n_K e^{\frac{-344}{197} \log n_K} \ll n_K^{-0.7}.
\]
Hence, $4\ell x^{-344/692\ell} = o(1)$ as $n_K \rightarrow \infty$. This proves the claim, which implies the condition on $\epsilon$ in \eqref{eqn:Exceptional_parameters} is non-empty for $\sL$ sufficiently large.

Now, let $1 \leq R \leq \tfrac{1}{2}\sL$ be arbitrary. By \cref{prop:FinalArgument}, for $x \geq e^{692\sL}$, we have that 
\begin{equation}
		\begin{aligned}
		 \frac{|G|}{|C|} \frac{S(x)}{e^{\epsilon} x} 
			& \leq 1 + \frac{x^{-(1-\beta_1)}}{\beta_1} + \frac{\log x}{e^{\epsilon}x} \sum_{\chi} \sumS_{\rho \neq \rho_1} |F(-\rho \log x)| 	+ O\big(\epsilon + x^{-186 R/692\sL} \big),
		\end{aligned}
		\label{eqn:SiegelZero_FirstStep}
	\end{equation} 
where $\sum^{\star}$ runs over non-trivial zeros $\rho \neq \rho_1$ of $L(s,\chi)$, counted with multiplicity, satisfying
\[
1 - R/\sL < \Re\{\rho\} < 1, \qquad |\Im\{\rho\}| \leq  \epsilon^{-1}. 
\]
Note that the $\beta_1$ term in \eqref{eqn:SiegelZero_FirstStep}	arises from bounding $F(-\sigma \log x)$ in \cref{lem:WeightChoice}(v) with $\sigma = \beta_1$. We further subdivide our arguments depending on the range of $\lambda_1$. 
\subsection{$\lambda_1$ very small ($\frac{2\eta \sL}{\log x} \leq \lambda_1 < \eta$)}
\label{subsec:SiegelZero_VerySmall}
Here select $\epsilon = \eta^2$ and $R = \min\{ \frac{1}{82} \log(c_1/\lambda_1), \tfrac{1}{2}\sL \}$ for some fixed sufficiently small $c_1 > 0$. Since $4\ell x^{-344/692 \ell} = o(1)$ as $\sL \rightarrow \infty$, it follows that this choice of $\epsilon$ satisfies \eqref{eqn:Exceptional_parameters} for $\sL$ sufficiently large depending only on $\eta$. 

Hence, by \cref{thm:ZeroRepulsion_Full}, these choices imply that the restricted sum $\sum^{\star}$ in \eqref{eqn:SiegelZero_FirstStep} is empty for $\sL$ sufficiently large depending only on $\eta$. Moreover, we see that
\[
x^{- 186 R/693\sL} \leq e^{-\frac{186}{82} \log(c_1/\lambda_1)} \ll \lambda_1^2 \ll \eta^2,
\]
as $x \geq e^{692\sL}$ and $186/82 > 2$. Further, we have that
\[
\frac{x^{-(1-\beta_1)}}{\beta_1} = e^{- \lambda_1 \log x/\sL}\{ 1 + O(\lambda_1/\sL)\} < 1 - \eta + O(\eta^2),
\]
since $\frac{2\eta \sL}{\log x} \leq \lambda_1 < \eta$ and $e^{-t} < 1 - t/2$ for $0 \leq t \leq 1$. Overall, we conclude that $S(x) < \{ 2 - \eta + O(\eta^2) \} \frac{|C|}{|G|} x$ for $x \geq e^{692 \sL}$. By \cref{lem:ReduceMainTheorems,lem:CharacterGroupBound}, we conclude that
\[
\pi_C(x,L/F) < \{ 2 - \eta +O\big(\eta^2 + \sL e^{1.4\sL} ( x^{-1/2} + e^{693\sL} x^{-1} \log x )\big) \} \frac{|C|}{|G|} \mathrm{Li}(x)
\]
for $x \geq e^{692 \sL}$. Hence, in this subcase, \cref{thm:BT_Weiss_sharp} (with no error term) follows for $x \geq e^{694.5 \sL}$ after fixing $\eta > 0$ sufficiently small and recalling $\sL$ is sufficiently large.
\subsection{$\lambda_1$ extremely small ($\lambda_1 < \frac{2\eta \sL}{\log x} \leq \eta$)}
\label{subsec:SiegelZero_ExtremelySmall}
Here select 
{\small
\[
\epsilon = 4\ell x^{-344/692\ell} \quad \text{and} \quad R = \min\Big\{ \frac{\sL}{81 \sL + 25 n_K \log(1/\epsilon)} \log\Big( \frac{c_1}{\lambda_1} \cdot \frac{\sL}{\sL+n_K \log(1/\epsilon)} \Big) , \frac{1}{2} \sL \Big\}
\]}%
for some sufficiently small $c_1 > 0$. Again, since $4\ell x^{-344/692 \ell} = o(1)$ as $\sL \rightarrow \infty$, it follows that $\epsilon < 1/4$ for $\sL$ sufficiently large so this choice of $\epsilon$ satisfies \eqref{eqn:Exceptional_parameters}.

Now, from our choice of $R$ and \cref{thm:ZeroRepulsion_Full}, the restricted sum in \eqref{eqn:SiegelZero_FirstStep} is empty. For the main term, observe for $\sL$ sufficiently large and $\eta > 0$ sufficiently small that
{\small
\[
\frac{x^{-(1-\beta_1)}}{\beta_1} < \Big( 1 - \frac{\lambda_1 \log x}{2 \sL} \Big)\Big(1 + \frac{\lambda_1}{\sL} \Big) \leq 1- \frac{\lambda_1 \log x}{3 \sL},
\]}%
as $\lambda_1 < \frac{2\eta \sL}{\log x}$ and $e^{-t} < 1-t/2$ for $0 \leq t \leq 1$. To bound the error term in \eqref{eqn:SiegelZero_FirstStep}, notice that
{\small\[
81 \sL + 25 n_K \log(1/\epsilon) \leq \frac{81}{692} \log x + \frac{344 \cdot 25 n_K }{692(82n_K + 162)} \log x < \frac{185.9}{692} \log x,
\]}%
by our choice of $\epsilon$ and $\ell$ and since $x \geq e^{693 \sL}$. Consequently, $R \geq \frac{692 \sL}{185.9 \log x} \log( \frac{c_1' \sL}{\lambda_1 \log x})$ for some sufficiently small $c_1' > 0$, implying
{\small
\[
x^{- 186 R /692\sL} \ll \Big( \frac{\lambda_1 \log x}{\sL}\Big)^{\frac{186}{185.9}} \ll \eta^{1/2000} \Big( \frac{\lambda_1 \log x}{\sL} \Big),
\]}%
since $\lambda_1 < \frac{2\eta \sL}{\log x}$ and $\frac{0.1}{185.9} < \frac{1}{2000}$. Combining these observations into \eqref{eqn:SiegelZero_FirstStep} implies that
{\small
\begin{align*}
\frac{|G|}{|C|} \frac{S(x)}{e^{\epsilon} x}  < 2 - \frac{\lambda_1 \log x}{3 \sL} + O\Big(\epsilon +  \eta^{1/2000} \cdot \frac{\lambda_1 \log x}{\sL}  \Big)    
< 2 - 100 \lambda_1 + O(\epsilon) 
\end{align*}}%
as $\eta$ is sufficiently small. Rearranging and substituting the choice of $\epsilon$ and $\ell$, we see that 
{\small
\[
S(x) < \Big\{ 2 - 100 \lambda_1 + O\big( n_K x^{-\frac{1}{166n_K + 327}} \big) \Big\} \frac{|C|}{|G|} x 
\]}%
for $x \geq e^{692 \sL}$. Now, if $x \geq e^{694.9 \sL}$ then, by \cref{lem:CharacterGroupBound}, we have that 
\[
	n_L e^{692 \sL}  x^{-1} \log x \ll n_K e^{693.4 \sL} x^{-1} \log x \ll n_K x^{-1.5/694.9} \log x \ll n_K x^{-1/(166n_K+327)}.
\]
Similarly, $n_L x^{-1/2} \ll n_K x^{-1/(166n_K+327)}$. Thus, by the previous inequality and \cref{lem:ReduceMainTheorems}, it follows that
{\small
\begin{equation}
\label{eqn:SiegelZero-Finish}
\pi_C(x,L/F) < \Big\{ 2-100 \lambda_1 +O\big( n_K x^{-\frac{1}{166n_K + 327}}  \big) \Big\} \frac{|C|}{|G|} \mathrm{Li}(x)
\end{equation}
for $x \geq e^{694.9 \sL}$. As $\delta_0$ in \eqref{def:sL} is sufficiently small, this completes the proof of \cref{thm:BT_Weiss_sharp} when an $\eta$-Siegel zero exists. 
\hfill \qed
\\
\begin{remark}
	~
	\begin{itemize}
		\item 	In \eqref{eqn:Exceptional_parameters}, we could instead take $B = 502$ and $E = 198$ to establish \eqref{eqn:SiegelZero-Finish} except with an error term of $O(n_K x^{-1/(208 n_K + 411)} )$. To improve the error term, we chose the largest values of $B$ and $E$ which did not reduce the valid range of $x$ in \cref{thm:BT_Weiss_sharp}.  This range of $x$ is limited by the case addressed in \cref{subsec:NonExceptional_LimitCase}.
		\item As stated in \cref{thm:BT_Weiss_sharp}, we obtain the sharper bound $\pi_C(x,L/F) <   2 \frac{|C|}{|G|}\mathrm{Li}(x)$ from \eqref{eqn:SiegelZero-Finish} with good effective lower bounds for $\lambda_1$. To see this, notice the error term in \eqref{eqn:SiegelZero-Finish} is $\ll \lambda_1^{1.001}$ provided
			\[
			x \gg \Big( \frac{c_1 n_K}{\lambda_1^{1.001}} \Big)^{166n_K + 327} =: x_1, 
			\]
			where $c_1 > 0$ is some absolute constant. If the above holds then \eqref{eqn:SiegelZero-Finish} becomes
			\[
			\pi_C(x,L/F) < \Big\{ 2-100 \lambda_1 +O(\lambda_1^{1.001} ) \Big\} \frac{|C|}{|G|} \mathrm{Li}(x).
			\] 
			As $\lambda_1 \leq \eta$, this implies $\pi_C(x,L/F) < 2\frac{|C|}{|G|}\mathrm{Li}(x)$ by fixing $\eta$ sufficiently small. Hence, any effective upper bound on $x_1$ translates to a range of $x$ where the sharper bound for $\pi_C(x, L/F)$ holds. 
		  From the proof of Theorem 1' in Stark \cite{Stark}, we have that $\lambda_1 \gg \min\{ g(n_K)^{-1}, D_K^{-1/n_K} \cQ^{-1/2n_K} \}$ where $g(n_K)$ equals $1$ if $K$ has a normal tower over $\Q$ and equals $(2n_K)!$ otherwise. If $n_K \leq 10$ and $D_K \cQ$ is sufficiently large then we have that 
		  \[
		  x_1 \ll (1/\lambda_1)^{167n_K + 328} \ll D_K^{167 + 328/n_K} \cQ^{84 + 164/n_K} \ll D_K^{495} \cQ^{248} \ll x,
		  \]
		   for $x$ satisfying \eqref{eqn:range_BT_Weiss_sharp}, as desired. 
		  Thus, we may assume $n_K \geq 10$ in which case we have that
		  \begin{align*}
					  x_1 & \ll n_K^{167 n_K} (1/\lambda_1)^{167n_K + 328}\\
					  &  \ll D_K^{167 + 328/n_K} \cQ^{84 + 164/n_K} n_K^{167n_K} + n_K^{167 n_K} g(n_K)^{167n_K+328}\\
  					  &  \ll D_K^{200} \cQ^{101} n_K^{167n_K} + n_K^{167 n_K} g(n_K)^{167n_K + 328}.			  	
		  \end{align*}
		  Therefore, if $K$ has a normal tower over $\Q$ or $(2n_K)! \ll D_K^{1/n_K} \cQ^{1/2n_K}$ then 
		  \[
		  x_1 \ll D_K^{200} \cQ^{101} n_K^{167 n_K} e^{O(n_K)} \ll  D_K^{200} \cQ^{101} n_K^{168 n_K} \ll x,
		  \]
		  for $x$ satisfying \eqref{eqn:range_BT_Weiss_sharp} and $D_K \cQ n_K^{n_K}$ sufficiently large. Otherwise, $g(n_K) \leq (2n_K)! \leq (2n_K)^{2n_K}$ which implies that 
		  \[
		  x_1  \ll D_K^{200} \cQ^{101} n_K^{167n_K} + n_K^{333n_K^2}
		  \]
		  unconditionally. Thus, imposing $x \gg n_K^{334n_K^2}$ in addition to \eqref{eqn:range_BT_Weiss_sharp} also yields the sharper estimate for $\pi_C(x,L/F)$. This completes all cases. 
		  
	\end{itemize}
\end{remark}

\section{Proof of \cref{thm:BT_Weiss_sharp}: $\eta$-Siegel zero does not exist}
\label{sec:Proof_BT_NoSiegelZero}
In this section, we assume $\lambda_1 \geq \eta$ for sufficiently small $\eta >0$ and we will  show \cref{thm:BT_Weiss_sharp} holds with no error term. Recall $\sL$ is sufficiently large depending only on $\eta$. Assume $\lambda^{\star}>0$ satisfies
\begin{equation}
\lambda^{\star} < \min\{ \lambda', \lambda_2\},
\label{eqn:LambdaStar}
\end{equation}
where $\lambda'$ and $\lambda_2$ are defined in \cref{subsec:LPI-notation}. Select 
\begin{equation}
B > 360, \qquad E = 198, \qquad \ell = 82n_K + 162, \qquad \epsilon = \eta^2,
\label{eqn:NonExceptional_parameters}
\end{equation}
and let $R = R(\eta)$ be sufficiently large. We claim these choices satisfy the assumptions of \cref{prop:FinalArgument}. Since $\sL$ is sufficiently large depending only on $\eta$, it suffices to show, for $x \geq e^{B\sL}$, that $4\ell x^{-E/B\ell} = o(1)$ as $\sL \rightarrow \infty$. We shall argue as in \cref{sec:Proof_BT_SiegelZeroExists}. If $n_K$ is bounded while $\sL \rightarrow \infty$ then this is immediate, so we may assume $n_K \rightarrow \infty$. By \eqref{eqn:sL_lowerbound}, notice that $\ell = 82n_K + 162 \leq \{ 196.8 + o(1)\} \frac{\sL}{\log n_K} \leq 197 \frac{\sL}{\log n_K}$ for $n_K$ sufficiently large. Thus, for $n_K$ sufficiently large and $x \geq e^{B \sL}$, we have that
\[
4\ell x^{-E/B\ell} \ll n_K e^{-198 \sL/\ell} \ll n_K e^{-\frac{198}{197} \log n_K} \ll n_K^{-1/197}. 
\]
Hence, $4\ell x^{-E/B\ell} = o(1)$ for $x \geq e^{B\sL}$, as $n_K \rightarrow \infty$. This proves the claim.

Therefore, by \cref{prop:FinalArgument}, it follows that
{\small
\[
\frac{|G|}{|C|} \frac{S(x)}{e^{\epsilon}x} \leq 1 + \frac{\log x}{e^{\epsilon}x} \sum_{\chi} \sumS_{\rho} |F(-\rho \log x)|  + O(\eta^2),
\]}%
for $x \geq e^{B \sL}$ and where the sum $\sum^{\star}$ runs over non-trivial zeros $\rho$ of $L(s,\chi)$, counted with multiplicity, satisfying $\beta > 1 - R/\sL$ and $|\gamma| \leq \eta^{-2}$.  For a non-trivial zero $\rho$ of a Hecke $L$-function, write $\rho = \beta + i\gamma =   1 - \frac{\lambda}{\sL}   + i \frac{\mu}{\sL}$. By \cref{lem:WeightChoice}, we see that
\[
\frac{\log x}{e^{\epsilon} x} |F(-\rho \log x)| \leq x^{-(1-\beta)} \leq e^{- B \lambda},
\]
since $x \geq e^{B \sL}$. Extracting $\rho_1$ and $\bar{\rho_1}$ (or simply $\rho_1$ if $\rho_1$ is real) from $\sum^{\star}$, we deduce by our choice of $\lambda^{\star}$ in \eqref{eqn:LambdaStar} that
{\small
\begin{equation}
\frac{|G|}{|C|} \frac{S(x)}{e^{\epsilon} x} \leq 1 + m(\rho_1) e^{-B \lambda_1} + \sum_{\chi} \,\,\, \sum_{\substack{ \lambda^{\star} \leq \lambda \leq R \\ |\gamma| \leq \eta^{-2}} } e^{-B \lambda}  + O(\eta^2),
\label{eqn:NoSiegelZero_Setup}
\end{equation}}%
where $m(\rho_1) = 2$ if $\rho_1$ is complex and $m(\rho_1) = 1$ if $\rho_1$ is real. To bound the remaining quantities,  we must select $\lambda^{\star}$ for which we further subdivide into cases. 

\subsection{$\lambda_1$ small ($\eta \leq \lambda_1 < 10^{-3}$)}

By \cref{thm:ZeroFreeRegions}, $\rho_1$ is a simple real zero attached to a real character $\chi_1$, implying $m(\rho_1) = 1$. Select $B = 361$ and choose $\lambda^{\star} = 0.2103 \log(1/\lambda_1)$, which satisfies \eqref{eqn:LambdaStar} by \cref{thm:ZeroRepulsion}. Arguing as in\footnote{Observe $361 > 297$ so the same estimates hold.} \cite[Section 10.1.2]{JT_AZ} and using \cref{thm:LFZD-LowLying}, we may conclude by \eqref{eqn:NoSiegelZero_Setup} that $S(x) < \{ 2- \eta + O(\eta^2) \} \frac{|C|}{|G|} x$ for $x \geq e^{361 \sL}$. As in the final arguments of \cref{subsec:SiegelZero_VerySmall}, we use \cref{lem:ReduceMainTheorems} to establish \cref{thm:BT_Weiss_sharp} for $x \geq e^{363 \sL}$. 

\subsection{$\lambda_1$ medium ($10^{-3} < \lambda_1 \leq 0.0875$)} One argues similar to the previous case with some minor changes. Namely, select $B = 593$ and choose $\lambda^{\star} = 0.44$, and follow \cite[Section 10.1.1]{JT_AZ} to deduce \cref{thm:BT_Weiss_sharp} for $x \geq e^{595 \sL}$. 

\subsection{$\lambda_1$ large ($\lambda_1 \geq 0.0875$)} 
\label{subsec:NonExceptional_LimitCase}
Select $B = 693$ and $\lambda^{\star} = 0.2866$ as per \cref{thm:ZeroFreeRegions}.  Noting $m(\rho_1) \leq 2$ unconditionally, one may argue similarly as per the previous cases and follow \cite[Section 11]{JT_AZ} to deduce \cref{thm:BT_Weiss_sharp} for $x \geq e^{694.9\sL}$. As $\delta_0$ in \eqref{def:sL} is sufficiently small, this yields the desired range of $x$ in \cref{thm:BT_Weiss_sharp}, completing the proof in all cases.  \hfill \qed


\section{Proof of Theorems \ref{thm:LT_1} and \ref{thm:LT_2}}
\label{sec:LT_proofs_1}

First, we state a slightly weaker (but more convenient) reformulation of \cref{thm:BT_Weiss}.
\begin{thm}
\label{thm:mk}
Let $L/F$ be a Galois extension of number fields with Galois group $G$, and let $C$ be any  conjugacy class of $G$. Let $H$ be an abelian subgroup of $G$ such that $H \cap C$ is non-empty, and let $K$ be the subfield of $L$ fixed by $H$.  Let $\mathcal{P}(L/K)$ be the set of rational primes $p$ such that there is a prime ideal $\kp$ of $K$ with $\kp\mid p$ and $\kp$ ramifies in $L$, and set
{\small
\[
M(L/K)=[L:K]D_K^{1/n_K}\prod_{p\in\mathcal{P}(L/K)}p.
\]}%
If $\log x\gg n_K\log(M(L/K)n_K)$, then $\pi_C(x,L/F) \ll \frac{|C|}{|G|} \mathrm{Li}(x)$.
\end{thm}
\begin{proof}
If $L/K$ is abelian, then \cite[Proposition 2.5]{MMS} states that
\[
\mathcal{Q}(L/K)\leq \Big([L:K]\prod_{p\in\mathcal{P}(L/K)}p\Big)^{2n_K}.
\]
Using the definition of $M(L/K)$, we see that \eqref{eqn:range_BT_Weiss} is
\[
\ll (D_K\mathcal{Q}(L/K)n_K^{n_K})^{246}\ll (n_K M(L/K))^{500n_K}.
\]
The claimed result now follows immediately from \cref{thm:BT_Weiss}.
\end{proof}

\subsection{Proof of \cref{thm:LT_1}}
Fix a newform $f$ (cf. Section 1) of even integral weight $k_f\geq2$, level $N_f$, and trivial nebentypus with integral Fourier coefficients, and fix an integer $a$.  For each prime $p$, we define $\omega_p=(a_f(p)^2-4p^{k_f-1})^{1/2}$.  We know from Deligne's proof of the Weil conjectures that $|a_f(p)|\leq 2p^{(k_f-1)/2}$ for all $p$, so $\Q(\omega_p)$ is an imaginary quadratic extension of $\Q$.  Set
\[
\pi_f(x,a;\ell)=\#\{\textup{$p\leq x$: $a_f(p)\equiv a\pmod\ell$ and $\ell$ splits in $\Q(\omega_p)$}\}.
\]
Let $\ell_1<\ell_2<\cdots<\ell_t$ be any $t$ odd primes, each less than $\exp(\frac{\log x}{2t})$.  By \cite[Corollary 4.2]{Wan}, if $t\sim(4/\log 2)\log\log x$, then
{\small
\begin{equation}
\label{eqn:pi_f_inequality}
\pi_f(x,a)\ll\sum_{j=1}^t \pi_f(x,a;\ell_j)+\frac{x}{(\log x)^2}\ll (\log\log x)\max_{1\leq j\leq t}\pi_f(x,a;\ell_j)+\frac{x}{(\log x)^2}.
\end{equation}}%
We proceed to bound $\pi_f(x,a;\ell)$, where $\ell\leq \exp((\log 2)(\log x)/(8\log\log x))$.

Let $\ell$ be prime, let $\mathbb{F}_{\ell}$ be the field of $\ell$ elements, and let $\mathrm{Frob}_p$ be the Frobenius automorphism of $\mathrm{Gal}(\overline{\Q}/\Q)$ at $p$.  For each $\ell$, there is a representation
\begin{equation}
\label{eqn:deligne_rep}
\rho_{f,\ell}:\mathrm{Gal}(\overline{\Q}/\Q)\to\mathrm{GL}_2(\mathbb{F}_\ell)
\end{equation}
which is unramified outside $N_f\ell $ such that for all primes $p\nmid N_f\ell$, we have that $\mathrm{tr}(\rho_{f,\ell}(\mathrm{Frob}_p))\equiv a_f(p)\pmod{\ell}$ and $\det(\rho_{f,\ell}(\mathrm{Frob}_p))\equiv p^{k_f-1}\pmod{\ell}$.  We have that $\rho_{f,\ell}$ is surjective for all but finitely many $\ell$.  Let $L=L_{\ell}$ be the subfield of $\overline{\Q}$ fixed by $\ker\rho_{f,\ell}$.  If $\ell$ is sufficiently large, then $L/\Q$ is a Galois extension, unramified outside of $N_f \ell$, whose Galois group is $G=\{g\in\mathrm{GL}_2(\mathbb{F}_{\ell}):\det g\in(\mathbb{F}_{\ell}^\times)^{k_f-1}\}$.

Define $C=\{\textup{$A\in G$: $\mathrm{tr}(A)\equiv a\pmod{\ell}$ and $\mathrm{tr}(A)^2-4\det(A)\in\mathbb{F}_{\ell}$ is a square}\}$.  Let $B$ denote the upper triangular matrices in $\mathrm{GL}_2(\mathbb{F}_{\ell})\cap G$, and let $L^B$ be the subfield of $L$ fixed by $B$.  Let $U$ be the unipotent elements of $B$, and let $L^U$ be the subfield of $L $ fixed by $U$.  Note that $U$ is a normal subgroup of $B$ and that $B/U\cong\mathrm{Gal}(L^U/L^B)$ is abelian.  Let $C'$ be the image of $C\cap B$ in $B/U$.  If $x$ is sufficiently large, then by \cite[Lemmas 2.7 and 4.3]{Zywina},
{\small
\[
\pi_f(x,a;\ell)\ll \pi_{C'}(x,L^U/L^B)+n_{L^B}\Big(\frac{\sqrt{x}}{\log x}+\log M(L^U/L^B)\Big).
\]}%
Applying \cref{thm:mk} to the Chebotarev prime counting functions for each conjugacy class in $C'$, we have that if $\log x\gg n_{L^B}\log(M(L^U/L^B)n_{L^B})$, then
{\small
\[
\pi_f(x,a;\ell)\ll \frac{|C'|}{|B/U|}\frac{x}{\log x}+n_{L^B}\Big(\frac{\sqrt{x}}{\log x}+\log M(L^U/L^B)\Big).
\]}%

By \cite[Lemma 4.4]{Zywina}, we have $|C'|/|B/U|\ll1/\ell$, $n_{L^B}\ll\ell$, and $\log M(L^U/L^B)\ll_{N_f}\log\ell$.  Combining all of our estimates, we find that
{\small
\begin{equation}
\label{eqn:upper_bound_pi_f}
\pi_f(x,a;\ell)\ll\frac{1}{\ell}\frac{x}{\log x}+\frac{\ell\sqrt{x}}{\log x}+\ell\log N_f\ell,\qquad \log x\gg \ell\log N_f\ell.
\end{equation}}%
Thus, taking $\ell\sim c' \log x / \log(N_f\log x)$ for some sufficiently small absolute constant $c'>0$,
{\small
\begin{equation}
\label{eqn:ell-adic_bound}
\pi_f(x,a;\ell)\ll\frac{x\log(N_f\log x)}{(\log x)^2}.
\end{equation}}%
Now, as before, let $t\in\Z$ satisfy $t\sim4/(\log 2)\log\log x$, and let $\ell_1<\ell_2<\cdots<\ell_t$ be $t$ consecutive primes with $\ell_1\sim c' \log x / \log(N_f\log x)$.  By the prime number theorem, $\ell_j\in[\ell_1,2\ell_1]$ for all $1\leq j\leq t$.  Therefore, if $c'$ is made sufficiently small, we have that
{\small
\begin{equation}
\label{eqn:pi_f_bound_2}
\max_{1\leq j\leq t}\pi_f(x,a;\ell_j)\ll\frac{x\log(N_f\log x)}{(\log x)^2}.
\end{equation}}%
\cref{thm:LT_1} now follows from inserting the inequality \eqref{eqn:pi_f_bound_2} into the inequality \eqref{eqn:pi_f_inequality}. 

\begin{remark}
Using the Cauchy-Schwarz and Polya-Vinogradov inequalities, V. K. Murty \cite[Page 304]{VKM} proved that
{\small
\begin{equation}
\label{eqn:vkm_1}
\pi_f(x,a)\ll \max_{\ell\in[y,2y]}\pi_f(x,a;\ell)+\Big(\frac{\pi_f(x,a)x\log y}{y}\Big)^{1/2}.
\end{equation}}%
Using \cite[Theorem 4.6]{VKM}, it is subsequently shown that if $\ell\in[y,2y]$ and $y=c'(\log x)/(\log\log x)^2$ for some sufficiently small absolute constant $c'>0$, then
{\small
\begin{equation}
\label{eqn:vkm_2}
\pi_f(x,a;\ell)\ll\frac{x(\log\log x)^2}{(\log x)^2}.
\end{equation}}%
It is then claimed in \cite{VKM} that \eqref{eqn:vkm_1} and \eqref{eqn:vkm_2} imply $\pi_f(x,a)\ll_{N_f}x(\log\log x)^2/(\log x)^2$.  It is not clear to us how to deduce this estimate for $\pi_f(x,a)$ using \eqref{eqn:vkm_1} and \eqref{eqn:vkm_2}. In particular, if $\pi_f(x,a) \gg x/(\log x)^2$, then the aforementioned choice of $y$ forces the secondary term in \eqref{eqn:vkm_1} to be $\gg x/(\log x)^{3/2}$. By inserting \eqref{eqn:vkm_2} into \eqref{eqn:pi_f_inequality} instead of \eqref{eqn:vkm_1}, one obtains the weaker statement \eqref{eqn:LT_111}.  The source of our improvement over \cite{VKM} stems solely from the $\log\log x$ savings over \eqref{eqn:vkm_2}, which can be seen from \eqref{eqn:ell-adic_bound}.
\end{remark}

\subsection{Proof of \cref{thm:LT_2}}
The proof of \cref{thm:LT_2} is nearly identical to the proof of \cite[Theorem 1.3(ii)]{Zywina} except that we use \cref{thm:BT_Weiss} to bound the ensuing Chebotarev prime counting function instead of using \cite[Theorem 2.1(ii)]{Zywina}.  The analytic details are very similar to the above proof of \cref{thm:LT_1}, but the particular Galois extension to which \cref{thm:BT_Weiss} is applied is different.  Following \cite[Section 5.2]{Zywina}, we apply \cref{thm:BT_Weiss} instead of \cite[Theorem 2.1(ii)]{Zywina}, which allows us to choose
\[
y=\frac{c}{h_k}\frac{\log x}{\log(\frac{D_k}{h_k}\log x)}
\]
(where $D_k$ is the absolute discriminant of $k$ and $h_k$ is the class number of $k$) for some sufficiently small absolute constant $c>0$.  This yields the claimed result.

\bibliographystyle{abbrv}
\bibliography{BT_for_CDT}
\end{document}